\documentclass[10pt]{article}
\usepackage{a4wide}
\usepackage{latexsym,stmaryrd}
\usepackage{amsmath,amssymb,amsthm,bm}
\usepackage[latin1]{inputenc}
\usepackage{t1enc, lmodern}
\usepackage{graphicx}
\usepackage{color}
\usepackage[margin=1in]{geometry}

\usepackage{enumerate}
\usepackage{graphicx}
\usepackage{url}

\usepackage{makeidx}

\usepackage{wrapfig}
\usepackage{amsmath}
\usepackage{amsfonts}
\usepackage{mathrsfs}
\usepackage{amssymb}
\usepackage{latexsym}
\usepackage{amsbsy}
\usepackage{color}
\usepackage{bbm}

\usepackage{mathrsfs}

\usepackage{wasysym}

\providecommand{\otherindexspace}[1]{}

\usepackage{amsthm}

\newtheorem{theorem}{Theorem}[section]

\newtheorem{lemma}[theorem]{Lemma}
\newtheorem{proposition}[theorem]{Proposition}
\newtheorem{remark}[theorem]{Remark}

\newtheorem{definition}[theorem]{Definition}

\newtheorem{corollary}[theorem]{Corollary}

\newtheorem{assumption}[theorem]{Assumption}

\numberwithin{equation}{section}

\def\esssup{\mathop{\mathrm{esssup}}\limits}
\def\essinf{\mathop{\mathrm{essinf}}\limits}

\def\s{\sigma}

\def\cal#1{\mathcal{#1}}

\def \H{\mathbb {H}}
\def \N{\mathbb {N}}
\def \R{\mathbb {R}}
\def \E{\mathbb {E}}
\def \F{\mathbb {F}}
\def \P{\mathbb {P}}
\def \mS{\mathcal{S}}
\def \mA{\mathcal{A}}
\def \mF{\mathcal{F}}
\def \mT{\mathcal{T}}
\def \mR{\mathcal{R}}
\def \mD{\mathcal{D}}
\def \mB{\mathcal{B}}
\def \mP{\mathcal{P}}

\usepackage{fancyhdr}

\newcommand{\ind}{\mathbf{1}}
\newcommand{\ov}{\overline}

\usepackage{graphics}
\usepackage{graphicx}

\DeclareGraphicsExtensions{.eps,.bmp,.jpg,.pdf,.mps,.png,.gif}

\makeatletter
\def\titre{\@title}
\makeatother

\begin{document}

\vspace{5cm}
\title{Numerical approximation of doubly reflected BSDEs with jumps and RCLL obstacles}
\author{Roxana DUMITRESCU\thanks{CEREMADE,
Universit\'e Paris 9 Dauphine, CREST  and  INRIA Paris-Rocquencourt, email: {\tt roxana@ceremade.dauphine.fr}. The research leading to these results has received funding from the R\'egion Ile-de-France. }
\and
C\'eline LABART\thanks{LAMA, 
Universit\'e de Savoie, 73376 Le Bourget du Lac, France and  INRIA Paris-Rocquencourt,
email: {\tt celine.labart@univ-savoie.fr}}}

\maketitle

\begin{abstract} We study a discrete time approximation scheme for the
  solution of a doubly reflected Backward Stochastic Differential Equation
  (DBBSDE in short) with jumps, driven by a Brownian motion and an independent
  compensated Poisson process. Moreover, we suppose that the obstacles are
  right continuous and left limited (RCLL) processes with predictable and
  totally inaccessible jumps and satisfy Mokobodzki's condition. Our main
  contribution consists in the construction of an implementable numerical
  sheme, based on two random binomial trees and the penalization method, which
  is shown to converge to the solution of the DBBSDE. Finally, we illustrate
  the theoretical results with some numerical examples in the case of general
  jumps.

\end{abstract}

\vspace{10mm}

\noindent{\bf Key words~:}  Double barrier reflected BSDEs, Backward
stochastic differential equations with jumps, Skorohod topology, numerical
sheme, penalization method.

\vspace{10mm}

\noindent{\bf MSC 2010 classifications~:} 60H10,60H35,60J75,34K28.

\newpage

\section{Introduction}

\quad In this paper, we study in the non-markovian setting a discrete time approximation scheme for the solution of a doubly reflected Backward Stochastic Differential Equation (DBBSDE in short) when the noise is given by a Brownian motion and a Poisson random process mutually independent. Moreover, the barriers are  supposed to be right-continuous  and left-limited  (RCLL in short) processes, whose jumps are arbitrary, they can be either predictable or inaccessible. The DBBSDE we solve numerically has the following form:
\begin{align}\label{eqintro}
  \left \lbrace \begin{tabular}{l}
    \mbox{(i) $Y_t=\xi_T+\int_t^T g(s,Y_s,Z_s,U_s)ds+(A_T-A_t)-(K_T-K_t)-\int_t^T Z_s
      dW_s-\int_t^T U_s d\tilde{N}_s$},\\
    \mbox{(ii) $\forall t \in [0,T]$, $\xi_t \le Y_t \le \zeta_t$ a.s.,}\\
    \mbox{(iii) $\int_0^T (Y_{t^-} -\xi_{t^-}) dA_t^c=0$ a.s.  and $\int_0^T (\zeta_{t^-} - Y_{t^-})
      dK_t^c=0$ a.s.} \\
   \mbox{(iv)} $\forall \tau \mbox{ predictable stopping time },\;\; \Delta
   A^d_{\tau}=\Delta A^d_{\tau}\ind_{Y_{\tau^-}=\xi_{\tau^-}} \mbox{ and}\;\; \Delta
    K^d_{\tau}=\Delta K^d_{\tau}\ind_{Y_{\tau^-}=\zeta_{\tau^-}}$.
  \end{tabular} \right.
\end{align} \quad Here, $A^c$ (resp. $K^c$) denotes the continuous part of $A$
(resp. $K$)  and $A^d$ (resp. $K^d$)  its discontinuous part, $\{W_t: 0 \leq t \leq T \}$ is a one dimensional
standard Brownian motion and $\{\Tilde{N}_t:=N_t-\lambda t, 0 \le t \le T\}$ is a compensated
Poisson process. Both processes are independent and they
are defined on the probability space $(\Omega,
\mathcal{F}_T,\mathbb{F}=\{\mathcal{F}_t\}_{0\le t \le T},P)$. The processes $A$ and $K$ have the role to keep the solution between the two obstacles $\xi$ and $\zeta$. Since we consider the general setting when the jumps of the obstacles can be either predictable or totally inaccessible, $A$ and $K$ are also discontinuous.\\

\quad  In the case of a Brownian filtration, non-linear backward stochastic
differential equations (BSDEs in short) were introduced by Pardoux and Peng
\cite{PP90}. One barrier reflected BSDEs have been firstly studied by El
Karoui et al in \cite{EKPPQ97}. In their setting, one of the components of the
solution is forced to stay above a given barrier which is a continuous adapted
stochastic process. The main motivation is the pricing of American options
especially in constrained markets. The generalization to the case of two
reflecting barriers has been carried out by Cvitanic and Karatzas in
\cite{CK96}. It is also well known that doubly reflected BSDEs are related to
Dynkin games and in finance to the pricing of Israeli options (or Game
options, see \cite{Kif13}).  The case of standard BSDEs with jump processes driven by a
compensated Poisson random measure was first considered by Tang and Li in
\cite{TL94}. The extension to the case of reflected BSDEs and one reflecting
barrier with only inaccessible jumps has been established by Hamad\`ene and
Ouknine \cite{HO03}. Later on, Essaky in \cite{Ess08} and Hamad\`ene and Ouknine
in \cite{HO13} have extended these results to a RCLL obstacle with predictable
and inaccessible jumps.
Results concerning existence and uniqueness of the solution for doubly
reflected BSDEs with jumps can be found  in \cite{CM08},\cite{DQS14},
\cite{HH06}, \cite{HW09} and \cite{EHO05}.\\

\quad  Numerical shemes for DBBSDEs driven by the Brownian motion and based on a
random tree method have been
proposed by Xu in \cite{Xu11} (see also \cite{MPX02} and \cite{PX11}) and, in the Markovian framework, by Chassagneux in
\cite{C09}. In the case of a filtration driven also by a Poisson process, some
results have been provided only in the non-reflected case. In \cite{BE08}, the
authors propose a scheme for Forward-Backward SDEs based on the dynamic
programming equation and in \cite{LMT07} the authors propose a fully
implementable scheme based
on a random binomial tree. This work extends the paper \cite{BDM01}, where
the authors prove a Donsker type theorem for BSDEs in the Brownian case.\\

\quad Our aim is to propose an implementable numerical method to approximate
the solution of DBBSDEs with jumps and RCLL obstacles \eqref{eqintro}. As for
standard BSDEs, the computation of conditional expectations is an important
issue. Since we consider reflected BSDEs, we also have to model the
constraints. To do this, we consider the following approximations
\begin{itemize}
\item we approximate the Brownian motion and the Poisson process by two independent random walks,
\item we introduce a sequence of penalized BSDEs to approximate the reflected BSDE.
\end{itemize}
 
These approximations enable us to provide a fully implementable scheme, called
{\it explicit penalized discrete scheme} in the following. We prove in Theorem \ref{main_thm} that the scheme weakly
converges to the solution of \eqref{eqintro}. 
Moreover, in order to prove the convergence of our sheme, we prove, in the
case of jump processes driven by a general Poisson random measure, that
the solutions of the penalized equations converge to the solution of the
doubly reflected BSDE in the case of a driver depending on the solution, which
was not the case in the previous literature (see \cite{EHO05}, \cite{HH06},
\cite{HW09}). This gives another proof for the existence of a solution of
  DBBSDEs with jumps and RCLL barriers. Our method is based on a combination of penalization, Snell
envelope theory, stochastic games, comparison theorem for BSDEs with jumps (see \cite{QS13}, \cite{QS14})  and a generalized monotonic theorem under the Mokobodzki's
condition. It extends \cite{LX07} to the case when the solution of the DBBSDE
also admits totally inaccessible jumps. Finally, we illustrate our theoretical
results with some numerical simulations in the case of general
jumps. We point out that the practical use of our scheme is restricted to
  low dimensional cases. Indeed, since we use a random walk to approximate the
  Brownian motion and the Poisson process, the complexity of the algorithm grows very fast in
  the number of time steps $n$ (more precisely, in $n^d$, $d$ being the
  dimension) and, as we will see in the numerical part, the penalization
  method requires many time steps to be stable.\\

\quad The paper is organized as follows: in Section 2 we introduce notation
and assumptions. In Section 3, we precise the discrete framework and give the
numerical scheme. In Section 4 we
provide the convergence by splitting the error : the error due to the approximation by penalization and the error due to the time discretization. Finally,
Section 5 presents some numerical examples, where the barriers contain
predictable and totally inaccessible jumps. In Appendix, we extend
the generalized  monotonic theorem and prove some technical results for discrete
BSDEs to the case of jumps. For the self-containment of the paper, we also
recall some recent results on BSDEs with jumps and reflected BSDEs.

\section{Notations and assumptions}

Although we propose a numerical scheme for
reflected BSDEs driven by a Brownian motion and a Poisson process, one part of
the proof of the convergence of our scheme is done in the general setting of
jumps driven by a Poisson random measure. Then, we first introduce
the general framework, in which we prove the convergence of a sequence of
   penalized
BSDEs to the solution of \eqref{eqintro}.

\subsection{General framework}

\subsubsection{Notation}

As said in Introduction, let $(\Omega, \F, P)$ be a probability space, and
${\mP}$ be the predictable $\sigma$-algebra on $[0,T] \times \Omega$. $W$ is a
one-dimensional Brownian motion and $N(dt,de)$ is a Poisson random measure,
independent of $W$, with compensator $\nu(de)dt$ such that $\nu$ is a
$\sigma$-finite measure on $\R^*$, equipped with its Borel field
${\mB}(\R^*)$.  Let $\tilde N(dt,du)$ be its compensated process.  Let
$\F = \{\mathcal{F}_t , 0\leq t \leq T \}$
be  the natural filtration associated with $W$ and $N$. \\

For each $T>0$, we use the following notations:
\begin{itemize}
\item
$L^2(\mathcal{F}_T)$  is the set of random variables $\xi$ which are  $\mathcal{F}_T
$-measurable and square integrable.

\item    $\H^{2}$ is the set of
real-valued predictable processes $\phi$ such that $\| \phi\|^2_{\H^{2}} := \E \left[\int_0 ^T \phi_t ^2 dt \right] < \infty.$

\item $L^2_\nu$ is the set of Borelian functions $\ell: \R^* \rightarrow \R$ such that  $\int_{\R^*}  |\ell(u) |^2 \nu(du) < + \infty.$

The set  $L^2_\nu$ is a
Hilbert space equipped with the scalar product
$\langle \delta    , \, \ell \rangle_\nu := \int_{\R^*} \delta(u) \ell(u) \nu(du)$ for all $\delta  , \, \ell \in L^2_\nu \times L^2_\nu,$
and the norm $\|\ell\|_\nu^2 :=\int_{\R^*}  |\ell(u) |^2 \nu(du).$

\item $\mathcal {B}(\R^2)$ (resp $\mathcal {B}(L^2_\nu ) $)  is the Borelian
$\sigma$-algebra on $\R^2$  (resp.  on $L^2_\nu $).

\item $\H_{\nu}^{2}$ is  the set of processes $l$ which are {\em predictable}, that is, measurable $$l : ([0,T]  \times \Omega \times \R^*,\; \mP \otimes {\cal B}(\R^*))  \rightarrow (\R\;,  \mB(\R)); \quad
(\omega,t,u) \mapsto l_t(\omega, u)
$$ such that $\| l \|^2_{\H_{\nu}^{2}} :=\E\left[\int_0 ^T \|l_t\|_{\nu}^2 \,dt \right]< \infty.$

\item   ${\cal S}^{2}$ is the set of real-valued RCLL adapted
processes $\phi$ such that $\| \phi\|^2_{\mathcal {S}^2} := \E(\sup_{0\leq t \leq T} |\phi_t |^2) <  \infty.$

\item $\mathcal {A}^2$ is the set of real-valued non decreasing RCLL predictable
processes $A$ with $A_0 = 0$ and $\E(A^2_T) < \infty$. 

\item  $\mT_{0}$ is the set of
stopping times $\tau$ such that $\tau \in [0,T]$ a.s

\item For $S$ in $\mT_{0}$,    $\mT_{S}$  is the set of
stopping times
$\tau$ such that $S \leq \tau \leq T$ a.s.

\end {itemize}

\subsubsection{Definitions and assumptions.}

We start this section by recalling the definition of a driver and a
Lipschitz driver. We also introduce DBBSDEs and our working assumptions.

\begin{definition}[Driver, Lipschitz driver]\label{defd}
A function $g$ is said to be a {\em driver} if
\begin{itemize}
\item
$g:  \Omega \times [0,T] \times \R^2 \times L^2_\nu \rightarrow \R $\\
$(\omega, t,y, z, \kappa(\cdot)) \mapsto  g(\omega, t,y,z,k(\cdot)) $
 is $ {\cal P} \otimes {\cal B}(\R^2)  \otimes {\cal B}(L^2_\nu)
- $ measurable,
\item $\|g(.,0,0,0)\|_{\infty}< \infty$.
\end{itemize}
A driver $g$ is called a {\em Lipschitz driver} if moreover there exists a
constant $ C_g \geq 0$ and a bounded, non-decreasing continuous function
$\Lambda$ with $\Lambda(0)=0$ such that $d \P \otimes dt$-a.s.\,,
for each $(s_1,y_1, z_1, k_1)$, $(s_2,y_2, z_2, k_2)$,
$$|g(\omega, s_1, y_1, z_1, k_1) - g(\omega, s_2, y_2, z_2, k_2)| \leq \Lambda(|s_2-s_1|)+C_g (|y_1 - y_2| + |z_1 - z_2| +   \|k_1 - k_2 \|_\nu).$$
\end{definition}

In the case of BSDEs with jumps, the coefficient $g$ must satisfy an additional assumption, which allows to apply the comparison theorem for BSDEs with jumps (see
Theorem \ref{QS13:th4.2}), which extends the result of \cite{Roy06}. More precisely, the driver $g$ satisfies the following assumption:

\begin{assumption}\label{hypo1} A Lipschitz driver $g$ is said to satisfy
  Assumption \ref{hypo1} if the following holds : $dP \otimes dt $ a.s. for
  each $(y,z,k_1,k_2) \in \R^2 \times (L^2_{\nu})^2$, we have
  \begin{align*}
    g(t,y,z,k_1)-g(t,y,z,k_2) \ge \langle \theta_t^{y,z,k_1,k_2},k_1-k_2 \rangle_{\nu},
  \end{align*}
  with
  \begin{align*}
    \theta:&\Omega \times [0,T] \times \R^2 \times (L^2_{\nu})^2 \longmapsto
    L^2_{\nu};\\
    &(\omega,t,y,z,k_1,k_2) \longmapsto \theta_t^{y,z,k_1,k_2}(\omega,\cdot)
  \end{align*}
  $\mP \otimes \mB(\R^2) \otimes \mB((L^2_{\nu})^2)$-measurable, bounded, and satisfying $dP
  \otimes dt \otimes \nu(du)$-a.s., for each $(y,z,k_1,k_2) \in \R^2 \times
  (L^2_{\nu})^2,$
  \begin{align*}
    \theta^{y,z,k_1,k_2}_t(u) \ge -1 \mbox{ and } |\theta^{y,z,k_1,k_2}_t(u)|
    \le \psi(u),
  \end{align*}
  where $\psi \in L^2_{\nu}$.
\end{assumption}


We now recall the ''Mokobodzki's condition'' which is essential in the case of doubly reflected BSDEs, since it ensures the existence of a solution. This condition essentially  postulates the existence of a quasimartingale between the barriers.

\begin{definition}[Mokobodzki's condition]\label{Moko}

Let $\xi$, $\zeta$ be in
  $\mS^2$. There exist two nonnegative RCLL supermartingales $H$ and $H'$ in
  $\mS^2$ such that
  \begin{align*}
\forall t \in [0,T],\;\; \xi_t \ind_{t < T} \le H_t-H'_t \le \zeta_t
\ind_{t<T} \mbox{ a.s.}
  \end{align*}
\end{definition}

\begin{assumption}\label{hypo2} $\xi$ and $\zeta$ are two adapted RCLL
  processes with $\xi_T=\zeta_T$ a.s., $\xi \in \mS^2$, $\zeta \in \mS^2$,
  $\xi_t \le \zeta_t$ for all $t \in [0,T]$, the Mokobodzki's condition holds
  and $g$ is a Lipschitz driver satisfying Assumption \ref{hypo1}.
\end{assumption}

We introduce the following general reflected BSDE with jumps and two RCLL
obstacles

\begin{definition}
  Let $T>0$ be a fixed terminal time and $g$ be a Lipschitz driver. Let $\xi$
  and $\zeta$ be two adapted RCLL processes with $\xi_T=\zeta_T$ a.s., $\xi
  \in \mS^2$, $\zeta \in \mS^2$, $\xi_t \le \zeta_t$ for all $t\in[0,T]$
  a.s. A process $(Y,Z,U,\alpha)$ is said to be a solution of the double barrier reflected
  BSDE (DBBSDE) associated with driver $g$ and barriers $\xi,\zeta$ if
  
  \begin{align}\label{eq0}
  \left \lbrace \begin{tabular}{l}
    \mbox{(i) $Y \in \mS^2$, $Z \in \H^2$, $U \in \H^2_{\nu}$ and $\alpha \in
      \mS^2$, where $\alpha=A-K$ with $A,K$ in $\mA^2$}\\
    \mbox{(ii) $Y_t=\xi_T+\int_t^T g(s,Y_s,Z_s,U_s)ds+(A_T-A_t)-(K_T-K_t)-\int_t^T Z_s
      dW_s-\int_t^T \int_{\R^*} U_s(e)\tilde{N}(ds,de)$},\\
    \mbox{(iii) $\forall t \in [0,T]$, $\xi_t \le Y_t \le \zeta_t$ a.s.,}\\
    \mbox{(iv) $\int_0^T (Y_{t^-} -\xi_{t^-}) dA_t=0$ a.s.  and $\int_0^T (\zeta_{t^-} - Y_{t^-})
      dK_t=0$ a.s.}
  \end{tabular} \right.
\end{align}

\end{definition}

\begin{remark}\label{rem14}
  Condition (iv) is equivalent to the following condition : if $K=K^c+K^d$ and
  $A=A^c+A^d$, where $K^c$ (resp. $K^d$) represents the continuous (resp. the
  discontinous) part of $K$ (the same notation holds for $A$), then
  \begin{align*}
    \int_0^T (Y_t-\xi_t)dA^c_t=0 \mbox{ a.s.},\;\;  \int_0^T
    (\zeta_t-Y_t)dK^c_t=0 \mbox{ a.s.}
  \end{align*}
  and
  \begin{align*}
    \forall \tau \in \mathcal{T}_0 \mbox{ predictable},\;\; \Delta
    A^d_{\tau}=\Delta A^d_{\tau}\ind_{Y_{\tau^-}=\xi_{\tau^-}} \mbox{ and}\;\; \Delta
    K^d_{\tau}=\Delta K^d_{\tau}\ind_{Y_{\tau^-}=\zeta_{\tau^-}}.
  \end{align*}
\end{remark}

\begin{theorem}(\cite[Theorem 4.1]{DQS14})\label{thm1}
  Suppose $\xi$ and $\zeta$ are RCLL adapted processes in $\mS^2$ such that
  for all $t\in [0,T]$, $\xi_t\le \zeta_t$ and Mokobodzki's condition
  holds (see Definition \ref{Moko}). Then, DBBSDE \eqref{eq0} admits a unique solution $(Y,Z,U,\alpha)$ in
  $\mS^2\times \H^2 \times \H^2_{\nu} \times \mA^2$.
\end{theorem}

\begin{remark} As said in \cite[Remark 4.3]{DQS14}, if for all $t \in ]0,T]$
  $\xi_{t^-}< \zeta_{t^-}$ a.s., \cite[Proposition 4.2]{DQS14} gives the
  uniqueness of $A,K \in (\mA^2)^2$.
\end{remark}

\begin{definition}[convergence in $J_1$-Skorokhod topology]
${\xi}^n$  is said to converge in probability (resp. in $L^2$) to $\xi$
      for the $J_1$-Skorokhod topology,
   if there exists a family
  $(\psi^n)_{n\in \N}$ of one-to-one random time changes (or stochastic changes of time scale) from $[0,T]$ to
  $[0,T]$  such that $\sup_{t \in [0,T]}|\psi^n(t)-t| \xrightarrow[n\rightarrow
  \infty]{}0$ almost surely and $\sup_{t \in [0,T]}
|{\xi}^n_{\psi^n(t)}-\xi_t|\xrightarrow[n \rightarrow \infty]{}0$ in
probability (resp. in $L^2$). Throughout the paper, we denote this convergence
$||\xi^n-\xi||_{J_1-\mathbb{P}} \rightarrow 0$ (resp. $||\xi^n-\xi||_{J_1-L^2} \rightarrow 0$).

\end{definition}

\subsection{Framework for our numerical scheme}

In order to propose an implementable numerical scheme we consider that the
Poisson random measure is simply generated by the jumps of a Poisson process. 
We consider a Poisson process $\{N_t: 0 \leq t \leq T\}$ with intensity $\lambda$ and jumps times $\{\tau_k: k=0,1,... \}$. The random measure is then
\begin{align*}
\Tilde{N}(dt,de)=\sum_{k=1}^{N_t} \delta_{\tau_k,1}(dt,de)-\lambda dt \delta_1(de)
\end{align*}
where $\delta_a$ denotes the Dirac measure  at the point $a$. In the
following, $\tilde{N}_t:=N_t-\lambda t$. Then, the unknown function $U_s(e)$ does
not depend on the magnitude $e$ anymore, and we write $U_s:=U_s(1)$.

In this particular case, \eqref{eq0} becomes:

 \begin{align}\label{DBBSDE1}
  \left \lbrace \begin{tabular}{l}
    \mbox{(i) $Y \in \mS^2$, $Z \in \H^2$, $U \in \H^2$ and $\alpha \in
      \mS^2$, where $\alpha=A-K$ with $A,K$ in $\mA^2$}\\
    \mbox{(ii) $Y_t=\xi+\int_t^T g(s,Y_s,Z_s,U_s)ds+(A_T-A_t)-(K_T-K_t)-\int_t^T Z_s
      dW_s-\int_t^T U_sd\tilde{N}_s$},\\
    \mbox{(iii) $\forall t \in [0,T]$, $\xi_t \le Y_t \le \zeta_t$ a.s.,}\\
    \mbox{(iv) $\int_0^T (Y_{t^-} -\xi_{t^-}) dA_t=0$ a.s.  and $\int_0^T (\zeta_{t^-} - Y_{t^-})
      dK_t=0$ a.s.}
  \end{tabular} \right.
\end{align}


In view of the proof of the convergence of the numerical scheme, we also
introduce the penalized version of \eqref{DBBSDE1}:
\begin{align}\label{eq_pen_scheme}
  Y^p_t=&\xi+\int_t^T g(s,Y^p_s,Z^p_s,U^p_s)ds +A^p_T-A^p_t
  -(K^p_T-K^p_t) -\int_t^T Z^p_s dW_s-\int_t^T U^p_s d\tilde{N}_s,
\end{align}
with $A^p_t:=p\int_0^t (Y^p_s-\xi_s)^-ds$ and $K^p_t:=p\int_0^t
(\zeta_s-Y^p_s)^- ds$, and $\alpha^p_t:=A^p_t-K^p_t$ for all $t\in [0,T]$.\\

\section{Numerical scheme} The basic idea is to approximate the Brownian
motion and the Poisson process by random walks based on the binomial tree
model. As explained in Section \ref{sect:mart_rep}, these approximations
enable to get a martingale representation whose coefficients, involving
conditional expectations, can be easily computed. Then, we approximate
  $(W,\tilde{N})$ in the penalized version of our DBBSDE (i.e. in
  \eqref{eq_pen_scheme}) by using these random walks. Taking conditional
  expectation and using the martingale representation leads to the {\it explicit
  penalized discrete scheme} \eqref{discrete_explicite}. In view of the proof
of the convergence of this explicit scheme, we
introduce an implicit intermediate scheme \eqref{discrete_implicite}.

\subsection{Discrete time Approximation}

We adopt the framework of \cite{LMT07}, presented below.

\subsubsection{Random walk approximation of $(W,\tilde{N})$}
For $n \in \mathbb{N}$, we introduce $\delta_n:=\frac{T}{n}$ and the regular grid
$(t_j)_{j=0,...,n}$ with step size $\delta_n$ (i.e. $t_j:=j \delta_n$) to
discretize $[0,T]$. In order to approximate $W$, we introduce the following
random walk
\begin{equation}
\begin{cases}
W_0^n=0\\
W_t^n=\sqrt{\delta_n} \sum_{i=1}^{[t/\delta_n]} e_i^n 
\end{cases}
\end{equation}
where $e_1^n, e_2^n,...,e_n^n$ are independent identically distributed random
variables with the following symmetric Bernoulli law:
$$P(e_1^n=1)=P(e_1^n=-1)=\frac{1}{2}.$$
To approximate $\tilde{N}$, we introduce a second random walk
\begin{equation}
\begin{cases}
\Tilde{N}_0^n=0\\
\Tilde{N}_t^n=\sum_{i=1}^{[t/\delta_n]}\eta_i^n 
\end{cases}
\end{equation}
where $\eta_1^n, \eta_2^n,...,\eta_n^n$ are independent and identically distributed random variables with law 
$$P(\eta_1^n=\kappa_n-1)=1-P(\eta_1^n=k_n)=\kappa_n $$
where $\kappa_n=e^{-\frac{\lambda}{n}}.$
We assume that both sequences $e_1^n,...,e_n^n$ and $\eta_1^n,
\eta_2^n,...,\eta_n^n$ are defined on the original probability space $(\Omega,
\mathbb{F}, P).$ The (discrete) filtration in the probability space
is $\mathbb{F}^n=\{\mathcal{F}_j^n: j=0,...,n \}$ with
$\mathcal{F}_0^n=\{\Omega,\emptyset \}$ and
$\mathcal{F}_j^n=\sigma\{e_1^n,...,e_j^n, \eta_1^n,...,\eta_j^n\}$ for
$j=1,...,n.$



The following result states the convergence of $(W^n,\tilde{N}^n)$ to
$(W,\tilde{N})$ for the $J_1$-Skorokhod topology, and the convergence of $W^n$
to $W$ in any $L^p$, $p\ge 1$, for the topology of uniform convergence on
$[0,T]$. We refer to \cite[Section 3]{LMT07} for more results on the
convergence in probability of $\mF^n$-martingales.

\begin{lemma}(\cite[Lemma3, (III)]{LMT07}, and \cite[Proof of Corollary
  2.2]{BDM01})\label{lem8} The couple $(W^n,\tilde{N}^n)$ converges in
  probability to $(W,\tilde{N})$ for the $J_1$-Skorokhod topology, and
  \begin{align*}
    \sup_{0\le t \le T} |W^n_t-W_t|\rightarrow 0\mbox{ as }n \rightarrow \infty
  \end{align*}
  in probability and in $L^p$, for any $1\le p < \infty$.
\end{lemma}

\subsubsection{Martingale representation}\label{sect:mart_rep}

Let $y_{j+1}$ denote a $\mF^n_{j+1}$-measurable random variable. As said in
\cite{LMT07}, we need a set of three strongly orthogonal martingales to
represent the martingale difference $m_{j+1}:=y_{j+1}-\E(y_{j+1}|\mF^n_j)$. We
introduce a third martingale increments sequence
$\{\mu^n_j=e^n_j\eta^n_j,j=0,\cdots,n\}$. In this context there exists a
unique triplet $(z_j,u_j,v_j)$ of $\mF^n_j$-random variables such that
\begin{align*}
  m_{j+1}:=y_{j+1}-\E(y_{j+1}|\mF^n_j)=\sqrt{\delta_n} z_j e^n_{j+1} + u_j
  \eta^n_{j+1} +v_j \mu^n_{j+1},
\end{align*}
and
\begin{align}\label{eq27}
  \left\{
  \begin{array}{l}
      z_j=\displaystyle{\frac{1}{\sqrt{\delta_n}}}\E(y_{j+1}e^n_{j+1}| \mF^n_j),\\
      u_j=\displaystyle{\frac{\E(y_{j+1}\eta^n_{j+1}|\mF^n_j)}{\E((\eta^n_{j+1})^2|\mF^n_j)}=\frac{1}{\kappa_n(1-\kappa_n)}\E(y_{j+1}\eta^n_{j+1}|
      \mF^n_j)},\\
      v_j=\displaystyle{\frac{\E(y_{j+1}\mu^n_{j+1}|\mF^n_j)}{\E((\mu^n_{j+1})^2|\mF^n_j)}=\frac{1}{\kappa_n(1-\kappa_n)}\E(y_{j+1}\mu^n_{j+1}|
      \mF^n_j)}\\
    \end{array}\right.
\end{align}

\begin{remark}(Computing the conditional expectations)\label{rem:cond_exp} Let $\Phi$ denote a
  function from $\R^{2j+2}$ to $\R$. We use the following formula to compute
  the conditional expectations
  \begin{align*}
    \E(\Phi(e^n_1,\cdots,e^n_{j+1},\eta^n_1,\cdots,\eta^n_{j+1})|\mF^n_j)=&\frac{\kappa_n}{2}\Phi(e^n_1,\cdots,e^n_{j},1,\eta^n_1,\cdots,\eta^n_{j},\kappa_n-1)\notag\\
    &+\frac{\kappa_n}{2}\Phi(e^n_1,\cdots,e^n_{j},-1,\eta^n_1,\cdots,\eta^n_{j},\kappa_n-1)\notag\\
    &+\frac{1-\kappa_n}{2}\Phi(e^n_1,\cdots,e^n_{j},1,\eta^n_1,\cdots,\eta^n_{j},\kappa_n)\notag\\
    &+\frac{1-\kappa_n}{2}\Phi(e^n_1,\cdots,e^n_{j},-1,\eta^n_1,\cdots,\eta^n_{j},\kappa_n).
  \end{align*}
\end{remark}

\subsection{Fully implementable numerical scheme}

In this Section we present two numerical schemes to approximate the solution
of the penalized equation \eqref{eq_pen_scheme}: the first one,
\eqref{discrete_implicite}, is an implicit intermediate scheme, useful for the proof of
convergence. We also introduce the main scheme \eqref{discrete_explicite}, which is
explicit. The implicit scheme \eqref{discrete_implicite} is not easy
to solve numerically, since it involves to inverse a function, as we will see
below. However, it plays an important role in the proof of the convergence of
the explicit scheme, that's
why we introduce it.\\

In both schemes, we approximate the barrier $(\xi_t)_t$ (resp. $(\zeta_t)_t$)
by $(\xi^n_j)_{j=0,\cdots,n}$ (resp. $(\zeta^n_j)_{j=0,\cdots,n}$). We also
introduce their continuous time versions:
\begin{align*}
\overline{\xi}^n_t:=\xi^n_{[t/\delta_n]},\;\;
\overline{\zeta}^n_t:=\zeta^n_{[t/\delta_n]}.
\end{align*}
These approximations satisfy



\begin{assumption}\label{hypo6}\hfill
  \begin{align*}
    & \mbox{(i) For some }r>2,\; \sup_{n \in \N} \max_{j\le n}
    \E(|\xi^n_j|^r)+\sup_{n \in \N} \max_{j\le n} \E(|\zeta^n_j|^r)+\sup_{t\le T} \E|\xi_t|^r+\sup_{t \le T} \E|\zeta_t|^r  < \infty\\
    & \mbox{(ii) $\overline{\xi}^n$ (resp $\overline{\zeta}^n$) converges in probability to $\xi$
      (resp. $\zeta$) for the $J_1$-Skorokhod topology.}
    \end{align*}
\end{assumption}


\begin{remark}\label{rem10} Assumption \ref{hypo6} implies that for all $t$ in $[0,T]$ $\overline{\xi}^n_{\psi^n(t)}$ (resp. $\overline{\zeta}^n_{\psi^n(t)}$) converges
  to $\xi_t$ (resp. $\zeta_t$) in $L^2$.
\end{remark}

\begin{remark}\label{rem15} Let us give different examples of barriers in $\mathcal{S}^2$
  satisfying Assumption \ref{hypo6}. In this Remark, $X$ represents either $\xi$ or $\zeta$.
  \begin{enumerate} 
  \item  $X$ satisfies the following SDE
  \begin{align*}
    X_t=X_0+\int_0^t
    b_{X}(X_{s^-})ds+\int_0^t\sigma_{X}(X_{s^-})dW_s+\int_0^t c_{X}(X_{s^-})d\tilde{N}_s
  \end{align*}
  where $b_{X}$, $\sigma_{X}$ and
  $c_{X}$ are Lipschitz functions. We approximate it by
  \begin{align*}
    \overline{X}^n_t=\overline{X}^n_0+\sum_{j=0}^{[t / \delta_n]-1}
    b_{X}(\overline{X}^n_{j \delta_n}) \delta_n+\int_0^t\sigma_{X}(\overline{X}^n_{s^-})dW^n_s+\int_0^t c_{X}(\overline{X}^n_{s^-})d\tilde{N}^n_s
  \end{align*} Since $(W^n,\tilde{N}^n)$
  converges in probability to $(W,\tilde{N})$ for the $J_1$-topology,
  \cite[Corollary 1]{SLO89} gives that $\overline{X}^n$ converges to $X$ in
  probability for the $J_1$-topology (for more details on the convergence of
  sequences of stochastic integrals on the space of RCLL functions endowed
  with the $J_1$-Skorokhod topology, we refer to \cite{JMP89}). Then,
  $\overline{X}^n$ satisfies Assumption \ref{hypo6} $(ii)$. We deduce from
  Doob and Burkh{\"o}lder-Davis-Gundy inequalities that $X$ and $\overline{X}^n$ satisfy
  Assumption \ref{hypo6} $(i)$ and that $X$
  belongs to $\mathcal{S}^2$.
\item $X$ is defined by $X_t:=\Phi(t,W_t,\tilde{N}_t)$, where $\Phi$ satisfies
  the following assumptions
  \begin{enumerate}
  \item $\Phi(t,x,y)$ is uniformly continuous in $(t,y)$ uniformly in $x$, i.e. there
    exist two continuous non decreasing functions $g_0(\cdot)$ and $g_1(\cdot)$ from $\R_+$ to
    $\R_+$ with linear growth and satisfying $g_0(0)=g_1(0)=0$
    such that
    \begin{align*}
      \forall\;(t,t',x,y,y'), \;\;|\Phi(t,x,y)-\Phi(t',x,y')|\le g_0(|t-t'|)+g_1(|y-y'|).
    \end{align*}
    We denote $a_0$ (resp. $a_1$) the constant of linear growth for $g_0$
    (resp. $g_1$) i.e.  $\forall \; (t,y)\in (\R_+)^2$,
    $0 \le g_0(t)+g_1(y)\le a_0(1+t)+a_1(1+y)$,
     \item $\Phi(t,x,y)$ is ``strongly'' locally Lispchitz in $x$ uniformly in $(t,y)$,
    i.e. there exists a constant $K_0$ and an integer $p_0$ such that
    \begin{align*}
      \forall\;(t,x,x',y),\;\;|\Phi(t,x,y)-\Phi(t,x',y)|\le K_0(1+|x|^{p_0}+|x'|^{p_0})|x-x'|.
    \end{align*}
  \end{enumerate} Then, $\forall (t,x,y)$ we have $|\Phi(t,x,y)|\le a_0|t|+a_1|y|+K_0(1+|x|^{p_0})|x|+|\Phi(0,0,0)|+a_0+a_1$. From this inequality, we prove that $X$ satisfies Assumption \ref{hypo6}
  $(i)$ by standard computations. Since $(\tilde{N}^n)$ converges in
  probability to $(\tilde{N})$ for the $J_1$-topology and $\lim_{n\rightarrow
    \infty} \sup_t|W^n_t-W_t|=0$ in $L^p$ for any $p$ (see Lemma \ref{lem8}), we get that
  $(X^n_t)_t:=(\Phi(\delta_n [t/\delta_n],W^n_t,\tilde{N}^n_t))_t$ converges in probability to $X$
  for the $J_1$-topology.
\end{enumerate}
  
\end{remark}

\subsubsection{Intermediate penalized implicit discrete scheme}
After the discretization of the penalized equation \eqref{eq_pen_scheme} on time intervals $[t_j, t_{j+1}]_{0 \leq j \leq n-1}$, we get the following discrete
backward equation. For all $j$ in $\{0,\cdots,n-1\}$

\begin{equation}\label{discrete_implicite_2}
\begin{cases}
y_{j}^{p,n}=y_{j+1}^{p,n}+g(t_j, y_j^{p,n}, z_j^{p,n}, u_j^{p,n}) \delta_n+a_j^{p,n}-k_j^{p,n}-(z_j^{p,n}\sqrt{\delta_n}e_{j+1}^n+u_j^{p,n}\eta_{j+1}^n+v_j^{p,n}\mu_{j+1}^n)\\
a_j^{p,n}=p \delta_n(y_j^{p,n}-\xi_j^n)^{-} \text{;   }k_j^{p,n}=p \delta_n(\zeta_j^n-y_j^{p,n})^{-},\\
y_n^{p,n}:=\xi^n_n.
\end{cases}
\end{equation}

Following \eqref{eq27}, the triplet $(z_j^{p,n}, u_j^{p,n}, v_j^{p,n})$ can be
computed as follows 

\begin{align*}
  \left\{ \begin{array}{l}
    \displaystyle{  z_j^{p,n}=\frac{1}{\sqrt{\delta_n}}\E(y^{p,n}_{j+1}e^n_{j+1}| \mF^n_j),}\\
      \displaystyle{u_j^{p,n}=\frac{1}{\kappa_n(1-\kappa_n)}\E(y^{p,n}_{j+1}\eta^n_{j+1}|
      \mF^n_j),}\\
      \displaystyle{v^{p,n}_j=\frac{1}{\kappa_n(1-\kappa_n)}\E(y^{p,n}_{j+1}\mu^n_{j+1}|
      \mF^n_j),}\\
    \end{array}\right.
\end{align*}
where we refer to Remark \ref{rem:cond_exp} for the computation of conditional
expectations. By taking the conditional expectation w.r.t. $\mF^n_j$ in
\eqref{discrete_implicite_2}, we get the following scheme, called {\it implicit
  penalized discrete scheme}: $y_n^{p,n}:=\xi^n_n$ and for $j=n-1,\cdots,0$ 
\begin{equation}\label{discrete_implicite}
  \begin{cases}
    y_{j}^{p,n}=(\Theta^{p,n})^{-1}(\E(y^{p,n}_{j+1}|\mF^n_j)),\\
    a_j^{p,n}=p \delta_n(y_j^{p,n}-\xi_j^n)^{-} \text{;   }k_j^{p,n}=p \delta_n(\zeta_j^n-y_j^{p,n})^{-},\\
    z_j^{p,n}=\displaystyle{\frac{1}{\sqrt{\delta_n}}}\E(y^{p,n}_{j+1}e^n_{j+1}| \mF^n_j),\\
    u_j^{p,n}=\displaystyle{\frac{1}{\kappa_n(1-\kappa_n)}}\E(y^{p,n}_{j+1}\eta^n_{j+1}|
    \mF^n_j),\\
  \end{cases}
\end{equation}
where $\Theta^{p,n}(y)=y-g(j\delta_n, y, z_j^{p,n},
u_j^{p,n})\delta_n-p\delta_n(y-\xi_j^n)^{-}+p\delta_n (\zeta_j^n-y)^{-}$.

We also introduce the continuous time version $(Y_t^{p,n}, Z_t^{p,n}, U_t^{p,n}, A_t^{p,n}, K_t^{p,n})_{ 0 \leq t \leq T}$ of the solution
to \eqref{discrete_implicite}:
\begin{align}\label{eq32}
  Y_t^{p,n}:=y_{[t/\delta_n]}^{p,n}, Z_t^{p,n}:=z_{[t/\delta_n]}^{p,n},
  U_t^{p,n}:=u_{[t/\delta_n]}^{p,n}, 
  A_t^{p,n}:=\sum_{i=0}^{[t/\delta_n]}a_i^{p,n}, K_t^{p,n}:=
  \sum_{i=0}^{[t/\delta_n]}k_i^{p,n}.
\end{align}
We also introduce ${\alpha}^{p,n}_t:={A}^{p,n}_t-{K}^{p,n}_t$, for
all $t \in [0,T]$.


\subsubsection{Main scheme}

As said before, the numerical inversion of the operator $\Theta^{p,n}$ is not
easy and is time consuming. If we replace $y^{p,n}_j$ by
$\E(y^{p,n}_{j+1}|\mF^n_j)$ in $g$, \eqref{discrete_implicite_2} becomes

\begin{equation}\label{discrete_explicite_2}
\begin{cases}
\ov{y}_{j}^{p,n}=\ov{y}_{j+1}^{p,n}+g(t_j, \E(\ov{y}^{p,n}_{j+1}|\mF^n_j), \ov{z}_j^{p,n}, \ov{u}_j^{p,n}) \delta_n+\ov{a}_j^{p,n}-\ov{k}_j^{p,n}-(\ov{z}_j^{p,n}\sqrt{\delta_n}e_{j+1}^n+\ov{u}_j^{p,n}\eta_{j+1}^n+\ov{v}_j^{p,n}\mu_{j+1}^n)\\
\ov{a}_j^{p,n}=p \delta_n(\ov{y}_j^{p,n}-\xi_j^n)^{-} \text{;   }\ov{k}_j^{p,n}=p \delta_n(\zeta_j^n-\ov{y}_j^{p,n})^{-},\\
\ov{y}_n^{p,n}:=\xi^n_n.
\end{cases}
\end{equation}

Now, by taking the conditional expectation in the above equation, we obtain:
\begin{equation}\label{expressiony}
\overline{y}_{j}^{p,n}=\E[\overline{y}_{j+1}^{p,n}|\mathcal{F}_j^n]+g(t_j, \E[\overline{y}_{j+1}^{p,n}|\mathcal{F}_j^n], \overline{z}_j^{p,n}, \overline{u}_j^{p,n})\delta_n+\ov{a}^{p,n}_j-\ov{k}^{p,n}_j.
\end{equation}


Solving this equation, we get the following scheme, called {\it explicit
  penalized scheme}: $\ov{y}_n^{p,n}:=\xi^n_n$ and for $j=n-1,\cdots,0$

\begin{equation}\label{discrete_explicite}
  \begin{cases}
    \ov{y}_{j}^{p,n}=\E[\ov{y}_{j+1}^{p,n}|\mathcal{F}^n_j]+g(t_j, \E(\ov{y}^{p,n}_{j+1}|\mF^n_j), \ov{z}_j^{p,n}, \ov{u}_j^{p,n}) \delta_n+\ov{a}_j^{p,n}-\ov{k}_j^{p,n},\\
    \ov{a}_j^{p,n}=\dfrac{p \delta_n}{1+p
      \delta_n}\left(\E[\overline{y}_{j+1}^{p,n}|\mathcal{F}_j^n]+\delta_n g(t_j,
      \E[\overline{y}_{j+1}^{p,n}|\mathcal{F}_j^n], \overline{z}_j^{p,n},
      \overline{u}_j^{p,n})-\xi_j^n\right)^{-},\\
    \ov{k}^{p,n}_j=\dfrac{p \delta_n}{1+p \delta_n}\left(\zeta_j^n-\E[\overline{y}_{j+1}^{p,n}|\mathcal{F}_j^n]-\delta_n g(t_j, \E[\overline{y}_{j+1}^{p,n}|\mathcal{F}_j^n], \overline{z}_j^{p,n}, \overline{u}_j^{p,n})\right)^{-}\\
    \displaystyle{\ov{z}_j^{p,n}=\frac{1}{\sqrt{\delta_n}}\E(\ov{y}^{p,n}_{j+1}e^n_{j+1}| \mF^n_j),}\\
    \ov{u}_j^{p,n}=\displaystyle{\frac{1}{\kappa_n(1-\kappa_n)}}\E(\ov{y}^{p,n}_{j+1}\eta^n_{j+1}|\mF^n_j).\\
\end{cases}
\end{equation}
  
\begin{remark}[Explanations on the derivation of the main scheme]
We give below some explanations concerning the derivation of the values of $\ov{a}_j^{p,n}$ and $\ov{k}_j^{p,n}$.   We  consider the following cases:
\begin{itemize}
\item[$\bullet$] If $ {\xi}_j^{n}<  \ov{y}_{j}^{p,n}<{\zeta}_j^{n}$, then by  $\eqref{discrete_explicite_2}$ we get $\ov{a}_j^{p,n}=\ov{k}_j^{p,n}=0$, which corresponds to\\ $\dfrac{p \delta_n}{1+p \delta_n}\left(\E[\overline{y}_{j+1}^{p,n}|\mathcal{F}_j^n]+\delta_n g(t_j,
      \E[\overline{y}_{j+1}^{p,n}|\mathcal{F}_j^n], \overline{z}_j^{p,n},
      \overline{u}_j^{p,n})-\xi_j^n\right)^{-}=\dfrac{p \delta_n}{1+p \delta_n}\left(\overline{y}_j^{p,n}-\xi_j^n\right)^{-}=0$ and\\ $\dfrac{p \delta_n}{1+p \delta_n}\left(\zeta_j^n-\E[\overline{y}_{j+1}^{p,n}|\mathcal{F}_j^n]-\delta_n g(t_j, \E[\overline{y}_{j+1}^{p,n}|\mathcal{F}_j^n], \overline{z}_j^{p,n}, \overline{u}_j^{p,n})\right)^{-}=\dfrac{p \delta_n}{1+p \delta_n}(\zeta_j^n-\ov{y}_j^{p,n})^-=0.$
 \item[$\bullet$] If $ {\xi}_j^{n} \geq  \ov{y}_{j}^{p,n}$, then  by  $\eqref{discrete_explicite_2}$ we have $\ov{a}_j^{p,n}=p \delta_n(\xi_j^n-\ov{y}_{j}^{p,n})$ and $\ov{k}_j^{p,n}=0$; we then replace $\ov{a}_j^{p,n}$  and $\ov{k}_j^{p,n}$ in $\eqref{expressiony}$ and we get  $\ov{a}_j^{p,n}=\dfrac{p \delta_n}{1+p
      \delta_n}\left(\E[\overline{y}_{j+1}^{p,n}|\mathcal{F}_j^n]+g(t_j,
      \E[\overline{y}_{j+1}^{p,n}|\mathcal{F}_j^n], \overline{z}_j^{p,n},
      \overline{u}_j^{p,n})\delta_n-\xi_j^n\right)^{-}.$ We also have $\dfrac{p \delta_n}{1+p \delta_n}\left(\zeta_j^n-\E[\overline{y}_{j+1}^{p,n}|\mathcal{F}_j^n]-\delta_n g(t_j, \E[\overline{y}_{j+1}^{p,n}|\mathcal{F}_j^n], \overline{z}_j^{p,n}, \overline{u}_j^{p,n})\right)^{-}=0$ and hence\\ $    \ov{k}^{p,n}_j=\dfrac{p \delta_n}{1+p \delta_n}\left(\zeta_j^n-\E[\overline{y}_{j+1}^{p,n}|\mathcal{F}_j^n]-\delta_n g(t_j, \E[\overline{y}_{j+1}^{p,n}|\mathcal{F}_j^n], \overline{z}_j^{p,n}, \overline{u}_j^{p,n})\right)^{-}$.
\item[$\bullet$] The case $ {\zeta}_j^{n} \leq  \ov{y}_{j}^{p,n}$ is symmetric to the one studied above: $ {\xi}_j^{n} \geq \ov{y}_{j}^{p,n}.$
\end{itemize}

\end{remark}

As for the implicit scheme, we define the continuous time version
$(\overline{Y}_t^{p,n},\overline{Z}_t^{p,n},\overline{U}_t^{p,n},
\overline{A}_t^{p,n}, \overline{K}_t^{p,n})_{0 \le t \le T}$ of the solution
to \eqref{discrete_explicite}:
\begin{align}\label{eq31}
\overline{Y}_t^{p,n}=\overline{y}_{[t/\delta_n]}^{p,n}, \quad  \overline{Z}_t^{p,n}=\overline{z}_{[t/\delta_n]}^{p,n}, \quad  \overline{U}_t^{p,n}=\overline{u}_{[t/\delta_n]}^{p,n}, \quad 
\overline{A}_t^{p,n}= \sum_{j=0}^{[t/\delta_n]} \overline{a}_j^{p,n} \quad 
\overline{K}_t^{p,n}= \sum_{j=0}^{[t/ \delta_n]} \overline{k}_j^{p,n}.
\end{align}
We also introduce $\ov{\alpha}^{p,n}_t:=\ov{A}^{p,n}_t-\ov{K}^{p,n}_t$, for
all $t \in [0,T]$.

\section{Convergence result}


The following result states the convergence
of $\ov{\Theta}^{p,n}:=(\ov{Y}^{p,n},\ov{Z}^{p,n},\ov{U}^{p,n},
  \ov{\alpha}^{p,n})$ to $\Theta:=({Y}, {Z}, {U},
  \alpha)$, the solution of the DBBSDE \eqref{DBBSDE1}.

\begin{theorem}\label{main_thm}
Assume that Assumptions \ref{hypo2} and \ref{hypo6} hold. The sequence
$(\ov{Y}^{p,n}, \ov{Z}^{p,n}, \ov{U}^{p,n})$ defined by \eqref{eq31}
converges to $(Y,Z,U)$, the solution of the DBBSDE \eqref{DBBSDE1}, in the
following sense: $\forall r \in [1,2[$
\begin{align}\label{conv1}
\lim_{p \rightarrow \infty} \lim_{n \rightarrow
  \infty}\left( \mathbb{E}\left[\int_0^T|\ov{Y}_{s}^{p,n}-Y_s|^2ds\right] +
\mathbb{E}\left[\int_0^T|\ov{Z}_{s}^{p,n}-Z_s|^rds\right]+ \mathbb{E}\left[\int_0^T|\ov{U}_{s}^{p,n}-U_s|^rds\right]\right)= 0.
\end{align}
Moreover, $\ov{Z}^{p,n}$ (resp. $\ov{U}^{p,n}$) weakly converges in $\H^2$ to $Z$
(resp. to $U$) and for $0 \leq t \leq T$, $\ov{\alpha}_{\psi^n(t)}^{p,n}$ converges weakly to  $\alpha_t$ in
$L^2(\mathcal{F}_T)$ as $n \rightarrow \infty$ and $p \rightarrow
\infty$. 
\end{theorem}

  In order to prove this result, we split the error in three terms, by
  introducing $\linebreak[4]\Theta^{p,n}_t:=({Y}_t^{p,n}, {Z}_t^{p,n}, {U}_t^{p,n},
  {\alpha}_t^{p,n})$, the solution of the implicit penalized discrete scheme \eqref{eq32} and $\Theta^{p}_t:=({Y}_t^{p}, {Z}_t^{p},
  {U}_t^{p}, {\alpha}_t^{p})$, the penalized version of
  \eqref{DBBSDE1}, defined by \eqref{eq_pen_scheme}. For the error on $Y$, we
  get
  \begin{align*}
\mathbb{E}[\int_0^T |\ov{Y}_{s}^{p,n}-Y_s|^2ds] \le 3\left(\mathbb{E}[\int_0^T
  |\ov{Y}_{s}^{p,n}-Y^{p,n}_s|^2ds] +\mathbb{E}[\int_0^T
  |Y_{s}^{p,n}-Y^p_s|^2ds] +\mathbb{E}[\int_0^T
  |Y_{s}^{p}-Y_s|^2ds]  \right),
\end{align*}
and the same splitting holds for $|\ov{Z}^{p,n}-Z|^r$ and $|\ov{U}^{p,n}-U|^r$. For the increasing processes, we have:
\begin{align}\label{Aconv}
\E[|\ov{\alpha}_{\psi^n(t)}^{p,n}-\alpha_{t}|^2] \leq
3\left(\E[|\ov{\alpha}_{\psi^n(t)}^{p,n}-\alpha_{\psi^n(t)}^{p,n}|^2]+\E[|\alpha_{\psi^n(t)}^{p,n}-\alpha^p_t|^2]+\E[|\alpha_t^p-\alpha_t|^2]\right).
\end{align}

The proof of Theorem \ref{main_thm} ensues from Proposition \ref{prop5},
Corollary \ref{cor1} and Proposition \ref{prop1}. Proposition \ref{prop5}
states the convergence of the error between $\ov{\Theta}^{p,n}$, the explicit
penalization scheme defined in \eqref{eq31}, and ${\Theta}^{p,n}$, the
implicit penalization scheme. It generalizes the results of \cite{PX11}. We
refer to Section \ref{sect:proof_conv_exp_imp}. Corollary \ref{cor1} states
the convergence (in $n$) of ${\Theta}^{p,n}$ to ${\Theta}^{p}$. This is based
on the convergence of a standard BSDE with jumps in discrete time setting
to the associated BSDE with jumps in continuous time setting, which is proved
in \cite{LMT07}. We refer to Section
\ref{sect:proof_conv_temps_discret}. Finally, Proposition \ref{prop1} proves
the convergence (in $p$) of the penalized BSDE with jumps ${\Theta}^{p}$ to
${\Theta}$, the solution of the DBBSDE \eqref{DBBSDE1}. In fact, we prove a
more general result in Section \ref{sect:proof_conv_pen}, since we show the
convergence of penalized BSDEs to \eqref{eq0} in the case of jumps driven
by a general Poisson random measure.

The rest of the Section is devoted to the proof of these results.

\subsection{Error between explicit and implicit penalization schemes}\label{sect:proof_conv_exp_imp}

We prove the convergence of the error between the explicit penalization scheme and the
implicit one. The scheme of the proof is inspired from \cite[Proposition 5]{PX11}.
\begin{proposition}\label{prop5} Assume Assumption \ref{hypo6} $(i)$ and $g$
  is a Lipschitz driver. We have
\begin{align*}
\lim_{n \rightarrow \infty} \sup_{0 \leq t \leq T}\left(\E[
|\overline{Y}_t^{p,n}-Y_t^{p,n}|^2]+\E[\int_0^T|\overline{Z}_s^{p,n}-Z_s^{p,n}|^2ds]+\E[\int_0^T|\overline{U}_s^{p,n}-U_s^{p,n}|^2 ds]\right) = 0.
\end{align*}
Moreover,
$\lim_{n\rightarrow \infty} (\overline{\alpha}^{p,n}_t-\alpha^{p,n}_t)=0$ in
$L^2(\mF_t)$, for $t\in [0,T]$.
\end{proposition}
 Recall that
\begin{align*}
  Y_t^{p,n}=y_{[t/\delta_n]}^{p,n}, Z_t^{p,n}=z_{[t/\delta_n]}^{p,n},
  U_t^{p,n}=u_{[t/\delta_n]}^{p,n}, 
  A_t^{p,n}=\sum_{i=0}^{[t/\delta_n]}a_i^{p,n}, K_t^{p,n}=
  \sum_{i=0}^{[t/\delta_n]}k_i^{p,n}.
\end{align*}
In a similar way we have defined the continuous time versions of $(\bar{y}^{p,n}, \bar{z}^{p,n},\bar{u}^{p,n},\bar{a}^{p,n},\bar{k}^{p,n})$, denoted by $(\bar{Y}^{p,n},\bar{Z}^{p,n},\bar{U}^{p,n},\bar{A}^{p,n},\bar{K}^{p,n})$. \\
\begin{proof}

By using the definitions of the implicit and explicit schemes
\eqref{discrete_implicite_2} and \eqref{discrete_explicite_2}, we obtain that:
\begin{align*} y_{j+1}^{p,n}-\overline{y}_{j+1}^{p,n}=&
  (y_{j}^{p,n}-\overline{y}_{j}^{p,n})+(g_p(t_j,
  \E[\overline{y}_{j+1}^{p,n}|\mathcal{F}_j^n],\overline{y}_j^{p,n},
  \overline{z}_j^{p,n}, \overline{u}_j^{p,n})-g(t_j,
  y_{j}^{p,n},y_{j}^{p,n},z_j^{p,n},u_j^{p,n})) \delta_n\\
  &+(z_{j}^{p,n}-\overline{z}_{j}^{p,n})e_{j+1}^n \sqrt \delta_n +
  (u_{j}^{p,n}-\overline{u}_{j}^{p,n})\eta_{j+1}^n+
  (v_{j}^{p,n}-\overline{v}_{j}^{p,n})\mu_{j+1}^n
\end{align*} where
$g_p(t,y_1,y_2,z,u)=g(t,y_1,z,u)+p(y_2-\overline{\xi}^n_t)^--p(\overline{\zeta}^n_t-y_2)^-$. It
implies that:
\begin{align*}
\E[(y_{j}^{p,n}-\overline{y}_{j}^{p,n})^2]=&
\E[(y_{j+1}^{p,n}-\overline{y}_{j+1}^{p,n})^2]-\E[(g_p(t_j, \E[\overline{y}_{j+1}^{p,n}|\mathcal{F}_j^n], \overline{y}_j^{p,n}, \overline{z}_j^{p,n}, \overline{u}_j^{p,n})-g_p(t_j, y_{j}^{p,n},y_{j}^{p,n},z_j^{p,n},u_j^{p,n}))^2]\delta_n^2\\
&- \E[(z_{j}^{p,n}-\overline{z}_{j}^{p,n})^2]\delta_n-\E[(u_{j}^{p,n}-\overline{u}_{j}^{p,n})^2](1-\kappa_n)\kappa_n-\E[(v_{j}^{p,n}-\overline{v}_{j}^{p,n})^2](1-\kappa_n)\kappa_n\\
&+2\E[(g_p(t_j,y_{j}^{p,n},y_{j}^{p,n}, z_j^{p,n}, u_j^{p,n})-g_p(t_j,\E[\overline{y}_{j+1}^{p,n}|\mathcal{F}_j^n],\overline{y}_j^{p,n},\overline{z}_j^{p,n}, \overline{u}_j^{p,n}))(y_j^{p,n}-\overline{y}_j^{p,n})]\delta_n.
\end{align*}
In the above relation, we take the sum over $j$ from $i$ to $n-1$. We have:
\begin{align*}
\E[(y_{i}^{p,n}-\overline{y}_{i}^{p,n})^2]&+\delta_n \sum_{j=i}^{n-1}\E[(z_{j}^{p,n}-\overline{z}_{j}^{p,n})^2]+(1-\kappa_n)\kappa_n\sum_{j=i}^{n-1}\E[ (u_{j}^{p,n}-\overline{u}_{j}^{p,n})^2]\\
&\leq 2\delta_n\sum_{j=i}^{n-1}\E[(g_p(t_j,y_{j}^{p,n},y_{j}^{p,n}, z_j^{p,n}, u_j^{p,n})-g_p(t_j,\E[\overline{y}_{j+1}^{p,n}|\mathcal{F}_j^n],\overline{y}_j^{p,n},\overline{z}_j^{p,n},\overline{u}_j^{p,n}))(y_j^{p,n}-\overline{y}_j^{p,n})].
\end{align*}
Let us introduce $f:y\longmapsto (y-\overline{\xi}^n_t)^-
-(\overline{\zeta}^n_t-y)^-$. We have
$g_p(t,y_1,y_2,z,u)=g(t,y_1,z,u)+pf(y_2)$. The last expectation of the previous
inequality can be written
\begin{align*}
\E[(g(t_j,y_{j}^{p,n}, z_j^{p,n}, u_j^{p,n})-g(t_j,\E[\overline{y}_{j+1}^{p,n}|\mathcal{F}_j^n],\overline{z}_j^{p,n}, \overline{u}_j^{p,n}))(y_j^{p,n}-\overline{y}_j^{p,n})+p(f(y_{j}^{p,n})-f(\overline{y}_j^{p,n}))(y_j^{p,n}-\overline{y}_j^{p,n})]
\end{align*}
Since $f$ is decreasing and $g$ is
Lipschitz, we obtain:
\begin{align*}
\E[(y_{i}^{p,n}-\overline{y}_{i}^{p,n})^2] &+ \delta_n\sum_{j=i}^{n-1}\E[ (z_{j}^{p,n}-\overline{z}_{j}^{p,n})^2]+(1-\kappa_n)\kappa_n\sum_{j=i}^{n-1}\E[ (u_{j}^{p,n}-\overline{u}_{j}^{p,n})^2]\\
&\leq
2\delta_n\sum_{j=i}^{n-1}\E\left[(C_g|y_{j}^{p,n}-\E[\overline{y}_{j+1}^{p,n}|\mathcal{F}_j^n]|+C_g|z_j^{p,n}-\overline{z}_j^{p,n}|+ C_g|u_j^{p,n}-\overline{u}_j^{p,n}|)|y_j^{p,n}-\overline{y}_j^{p,n}|\right].
\end{align*}
Consequently, by applying the inequality $2ab \leq a^2+b^2$  for $a=C_g |y_j^{p,n}-\overline{y}_j^{p,n}| \sqrt {2\delta_n};\,\,b= \displaystyle{\sqrt {\frac{\delta_n}{2}}} |z_j^{p,n}-\overline{z}_j^{p,n}|$ and $a=C_g |y_j^{p,n}-\overline{y}_j^{p,n}| \sqrt {2}\displaystyle{ \frac{\delta_n}{\sqrt {\kappa_n(1-\kappa_n)}}};\,\,b= \sqrt {\frac{\kappa_n(1-\kappa_n)}{2}} |u_j^{p,n}-\overline{u}_j^{p,n}|\,\,\,$   we get that:
\begin{align*}
\E[(y_{i}^{p,n}-\overline{y}_{i}^{p,n})^2] & +\delta_n\sum_{j=i}^{n-1}\E[ (z_{j}^{p,n}-\overline{z}_{j}^{p,n})^2]+(1-\kappa_n)\kappa_n\sum_{j=i}^{n-1}\E[ (u_{j}^{p,n}-\overline{u}_{j}^{p,n})^2]\\
&\leq 2\delta_n C_g^2\sum_{j=i}^{n-1} \E[
(y_j^{p,n}-\overline{y}_j^{p,n})^2]+\frac{\delta_n}{2}\sum_{j=i}^{n-1}\E[
(z_j^{p,n}-\overline{z}_j^{p,n})^2]+\frac{2C_g^2 \delta_n^2}{\kappa_n(1-\kappa_n)}\sum_{j=i}^{n-1}\E[
(y_j^{p,n}-\overline{y}_j^{p,n})^2]\\
&+\frac{(1-\kappa_n)\kappa_n}{2} \sum_{j=i}^{n-1}\E[(u_j^{p,n}-\overline{u}_j^{p,n})^2]+2C_g\delta_n\mathbb{E}[\sum_{j=i}^{n-1}|y_j^{p,n}-\overline{y}_j^{p,n}||y_j^{p,n}-\mathbb{E}[\overline{y}_{j+1}^{p,n}|\mathcal{F}_j^n]|].
\end{align*}
Now, since
$\overline{y}_j^{p,n}-\mathbb{E}[\overline{y}_{j+1}^{p,n}|\mathcal{F}_j^n]=g_p(t_j,\mathbb{E}[\overline{y}_{j+1}^{p,n}|\mathcal{F}_j^n],\overline{z}_j^{p,n},\overline{u}_j^{p,n})\delta_n$,
the last term is dominated by
\begin{equation*}
\delta_n\sum_{j=i}^{n-1}(2C_g+1)\mathbb{E}[(y_j^{p,n}-\overline{y}_j^{p,n})^2]+C_g^2\delta_n^3\sum_{j=i}^{n-1}\mathbb{E}[g_p(t_j,\mathbb{E}[\overline{y}_{j+1}^{p,n}|\mathcal{F}_j^n],\overline{y}_j^{p,n},\overline{z}_j^{p,n},\overline{u}_j^{p,n})^2].
\end{equation*}
Using the definition of $g_p$ yields
\begin{align*}
  g_p(t_j,\mathbb{E}[\overline{y}_{j+1}^{p,n}|\mathcal{F}_j^n],\overline{y}_j^{p,n},\overline{z}_j^{p,n},\overline{u}_j^{p,n})&\le
  |g(t_j,\mathbb{E}[\overline{y}_{j+1}^{p,n}|\mathcal{F}_j^n],\overline{z}_j^{p,n},\overline{u}_j^{p,n})|+p  (|\overline{y}_j^{p,n}|+|\xi^n_j|+|\zeta^n_j|),\\
  &\le |g(t_j,0,0,0)|+C_g(|\mathbb{E}[\overline{y}_{j+1}^{p,n}|\mathcal{F}_j^n]|+|\overline{z}_j^{p,n}|+|\overline{u}_j^{p,n}|)+p (|\overline{y}_j^{p,n}|+|\xi^n_j|+|\zeta^n_j|).
\end{align*}
We get
\begin{align*}
  \delta_n^3\sum_{j=i}^{n-1}\mathbb{E}[g_p(t_j,\mathbb{E}[\overline{y}_{j+1}^{p,n}|\mathcal{F}_j^n],\overline{y}_j^{p,n},\overline{z}_j^{p,n},\overline{u}_j^{p,n})^2]\le&
  C_0\delta_n^2(\delta_n
  \sum_{j=i}^{n-1}|g(t_j,0,0,0)|^2+\delta_n\sum_{j=i}^{n-1}|\overline{z}_j^{p,n}|^2+\delta_n\sum_{j=i}^{n-1}|\overline{u}_j^{p,n}|^2)\\
  &+C_0
  (p\delta_n)^2 (\max_{j}\E(|\xi^n_j|^2)+\max_{j}\E(|\zeta^n_j|^2))\\
  &+C_0\delta_n^2(1+p^2)\max_{j}\E(|\overline{y}_j^{p,n}|^2)
\end{align*}
where $C_0$ denotes a generic constant depending on $C_g$.
Since $\displaystyle{\frac{\delta_n}{(1-\kappa_n)\kappa_n}=\frac{1}{\lambda} \frac{\lambda
  \delta_n}{(1-e^{-\lambda \delta_n})e^{-\lambda \delta_n}}}$ and $\displaystyle{e^x \le \frac{x
  e^{2x}}{e^x-1}\le e^{2x}}$, we get
$\displaystyle{\frac{\delta_n}{(1-\kappa_n)\kappa_n}\le \frac{1}{\lambda}e^{2\lambda T}}$.
Hence, for $\delta_n$ small enough such that 
$(3+2p+2C_g+2C_g^2(1+\frac{1}{\lambda}e^{2 \lambda T}))\delta_n <1$, Lemma
\ref{lem4} enables to write:

\begin{align}\label{eq36}
  \E[(y_{i}^{p,n}-\overline{y}_{i}^{p,n})^2]+&\frac{\delta_n}{2}\E[
\sum_{j=i}^{n-1}(z_j^{p,n}-\overline{z}_j^{p,n})^2]+\frac{1}{2}(1-\kappa_n)\kappa_n\E[
\sum_{j=i}^{n-1}(u_j^{p,n}-\overline{u}_j^{p,n})^2]\notag \\
&\leq \left(1+2C_g+2C_g^2+\frac{2C_g^2\delta_n}{(1-\kappa_n)\kappa_n}\right)\delta_n \mathbb{E}[
\sum_{j=i}^{n-1}(y_j^{p,n}-\overline{y}_j^{p,n})^2]+C_1(p)\delta_n^2,
\end{align}
where
$C_1(p)=C_0(\|g(\cdot,0,0,0)\|^2_{\infty}+p^2(\sup_n\max_{j}\E|\xi^n_j|^2+\sup_{n}\max_j\E|\zeta^n_j|^2)+(1+p^2)K_{\mbox{Lem.}
\ref{lem4}})$, $K_{\mbox{Lem.}\ref{lem4}}$ denotes the constant appearing
in Lemma \ref{lem4}. Discrete Gronwall's Lemma (see \cite[Lemma 3]{PX11}) gives 
\begin{align*}
\sup_{i\leq n}\mathbb{E}[(y_{i}^{p,n}-\overline{y}_{i}^{p,n})^2]\leq
C_1(p)\delta_n^2e^{(1+2C_g+2C_g^2(1+\frac{1}{\lambda}e^{2 \lambda T}))T}.
\end{align*}
Since $\delta_n \le T$, $(1-\kappa_n)\kappa_n \ge \lambda \delta_n e^{-2\lambda
  T}$, and Equation \eqref{eq36} gives
\begin{align*}
  \E[\int_0^T|\overline{Z}_s^{p,n}-Z_s^{p,n}|^2ds]+
  \E[\int_0^T|\overline{U}_s^{p,n}-U_s^{p,n}|^2 ds] \le C'_1(p) \delta_n^2,
\end{align*}
where $C'_1(p)$ is another constant depending on $C_g$, $\lambda$, $ T$ and $C_1(p)$.
It remains to prove the convergence for the increasing processes. We have
\begin{align*}
  \ov{A}^{p,n}_t-\ov{K}^{p,n}_t=\ov{Y}^{p,n}_0-\ov{Y}^{p,n}_t-\int_0^t
  g(s,\ov{Y}^{p,n}_s,\ov{Z}^{p,n}_s,\ov{U}^{p,n}_s)ds + \int_0^t \ov{Z}^{p,n}_s dW^n_s +\int_0^t
  \ov{U}^{p,n}_s d\tilde{N}^n_s,\\
  A^{p,n}_t-K^{p,n}_t=Y^{p,n}_0-Y^{p,n}_t-\int_0^t
  g(s,Y^{p,n}_s,Z^{p,n}_s,U^{p,n}_s)ds + \int_0^t Z^{p,n}_s dW^n_s +\int_0^t
  U^{p,n}_s d\tilde{N}^n_s.\\
\end{align*}
Using the Lispchitz property of $g$ and the convergence of
$(\ov{Y}^{p,n}_s-Y^{p,n}_s,\ov{Z}^{p,n}_s-Z^{p,n}_s,\ov{U}^{p,n}_s-U^{p,n}_s)$, we get the result.
\end{proof}

 \subsection{Convergence of the discrete time setting to the continuous time
   setting}\label{sect:proof_conv_temps_discret}
 
The following Proposition ensues from \cite{LMT07}. 
 \begin{proposition}\label{discreteT}Let $g$ be a Lipschitz driver and assume
   that Assumption \ref{hypo6} $(ii)$ holds.
   For any $p \in \mathbb{N}^*$, the sequence $(Y_t^{p,n}, Z_t^{p,n}, U_t^{p,n})$ converges to $(Y_t^p,Z_t^p,U_t^p)$ in the following sense:
   \begin{align}\label{conv}
     \lim_{n \rightarrow \infty} \left(||Y^{p,n}-Y^p||_{J_1-L^2}^2+\mathbb{E}[\int_{0}^T|Z_s^{p,n}-Z_s^p|^2ds+
     \int_0^T|U_s^{p,n}-U_s^p|^2ds]\right) = 0.
   \end{align}
  
 \end{proposition}

 \begin{proof}
For a fixed $p$, we have the following:
\begin{align}\label{decomp}
Y^{p,n}-Y^p=(Y^{p,n}-Y^{p,n,q})+(Y^{p,n,q}-Y^{p,\infty,q})+(Y^{p,\infty,q}-Y^{p}).
\end{align} where $(Y^{p,\infty,q}, Z^{p,\infty,q}, U^{p,\infty,q} )$ is the
Picard approximation of $(Y^{p}, Z^{p}, U^{p} )$ and $(Y^{p,n,q}, Z^{p,n,q},
U^{p,n,q})$ represents the continuous time version of the discrete Picard
approximation of $(y_k^{p,n}, z_k^{p,n},u_k^{p,n})$, denoted by $(y_k^{p,n,q},
z_k^{p,n,q},u_k^{p,n,q})$. Note that  $(y_k^{p,n,q+1}, z_k^{p,n,q+1},u_k^{p,n,q+1})$
is defined inductively as the solution of the backward recursion given by
\cite[Eq. (3.16)]{LMT07}, for the penalized driver $g_n(\omega,t,y,z,u):=g(\omega,t,
y,z,u)+p(y-\overline{\xi}^n_t(\omega))^-
-p (\overline{\zeta}^n_t(\omega)-y)^-$. Since $\overline{\xi}^n$ and $\overline{\zeta}^n$ satisfy
Assumption \ref{hypo6} $(ii)$,  $(g_n(\omega,\cdot,\cdot,\cdot,\cdot))_n$ converges
uniformly to $g(\omega,\cdot,\cdot,\cdot,\cdot)+p(y-\xi_{t}(\omega))^-
-p (\zeta_{t}(\omega)-y)^-$ almost surely up to a
subsequence (i.e. $g_n$
satisfies \cite[Assumption (A')]{LMT07}).\\
Now, by using \eqref{decomp}, \cite[Proposition 1]{LMT07}, \cite[Proposition
3]{LMT07} and \cite[Eq. (3.17)]{LMT07}, one
can easily show that \eqref{conv} holds.
\end{proof}

 The following Corollary ensues from Proposition \ref{discreteT}.
 \begin{corollary}\label{cor1} Let $g$ be a Lipschitz driver, $\xi$ and
   $\zeta$ belong to $\mathcal{S}^2$, $\psi^n$ is the random mapping introduced
   in Proposition \ref{discreteT} and assume
   that Assumption \ref{hypo6} holds. For any $p \in \mathbb{N}^*$, the sequence $(Y_t^{p,n}, Z_t^{p,n}, U_t^{p,n})$ converges to $(Y_t^p,Z_t^p,U_t^p)$ in the following sense:
   \begin{align*}
     \lim_{n \rightarrow \infty}\mathbb{E}[\int_0^T |Y_s^{p,n}-Y_s^p|^2 ds +\int_{0}^T|Z_s^{p,n}-Z_s^p|^2ds+
     \int_0^T|U_s^{p,n}-U_s^p|^2ds] = 0,
   \end{align*}
   Moreover, $A^{p,n}$ (resp. $K^{p,n}$) converges
  to $A^p$ (resp. $K^p$) when $n$ tends to infinity in $L^2$ for the $J_1$-Skorokhod topology.
 \end{corollary}
  
\begin{proof}

Note that:
\begin{align*}
\int_0^T|Y_s^{p,n}-Y_s^p|^2ds\leq2 \int_0^T|Y_s^{p,n}-Y_{\eta^n(s)}^p|^2dt+2\int_0^T|Y_{\eta^n(s)}^p-Y_s^p|^2ds,
\end{align*}
where $\eta^n(s)$ represents the inverse of $\psi^n(s).$

Proposition \ref{discreteT} gives that the first term in the right-hand side
converges to 0. Concerning the second term, $s \mapsto Y_s^{p}$ is continuous
except at the times at which the Poisson process jumps. Consequently, $
Y_{\eta^n(s)}^{p}$ converges to $Y_s^{p}$ for almost every $s$ and as $Y^{p}$
belongs to $\mathcal{S}^2$, we get that
$\mathbb{E}[\int_0^T|Y_{\eta^n(s)}^p-Y_s^p|^2ds] \rightarrow 0$ when
$n\rightarrow \infty$.

 Now, remark that we can rewrite $A_{t}^{p,n}$ and $A_t^p$ as follows:
\begin{align}
A_t^{p,n}=p\int_0^t(Y_s^{p,n}-\overline{\xi}_s^n)^{-}ds \qquad A_t^{p}=p\int_0^t(Y_s^{p}-\xi_s)^{-}ds. 
\end{align}
Then
\begin{align*}
\sup_{t \in [0,T]} |A_{\psi^n(t)}^{p,n}-A^p_t|&=\sup_{t \in [0,T]}
|A_t^{p,n}-A^p_{\eta^n(t)}|\\
&=\sup_{k\in \{0,\cdots,n\}} |A^{p,n}_{t_k}-A^p_{t_k}|+\sup_{k\in
  \{0,\cdots,n\}}\sup_{t\in [t_k,t_{k+1}]}|A^p_{t_k}-A^p_{\eta^n(t)}|.
\end{align*}
since $\xi$ and $Y^p$ belong to $\mathcal{S}^2$, we get that the second term
in the right hand side tends to $0$ in $L^2$ when
$n\rightarrow \infty$.
\begin{align*}
  \sup_{k\in \{0,\cdots,n\}} |A^{p,n}_{t_k}-A^p_{t_k}|&\le 
  p\int_0^{T}|Y_s^{p,n}-Y^p_s| + |\overline{\xi}_s^n-\xi_s|ds.
\end{align*} Since $\lim_{n \rightarrow \infty}\mathbb{E}[\int_0^T
|Y_s^{p,n}-Y_s^p|^2 ds]=0$, $\lim_{n \rightarrow \infty}
\E|\overline{\xi}^n_s-\xi_{\eta(s)}|^2=0$ (see Remark \ref{rem10}) and $\lim_{n \rightarrow
  \infty}\mathbb{E}[\int_0^T|\xi_{\eta^n(s)}-\xi_s|^2ds]=0$ ($\xi$ is RCLL,
its jumps are countable), we get that $\sup_{k\in \{0,\cdots,n\}}
|A^{p,n}_{t_k}-A^p_{t_k}|$ converges to $0$ in $L^2$ in $n$, which ends the proof.

\end{proof}

\subsection{Convergence of the penalized BSDE to the reflected
  BSDE}\label{sect:proof_conv_pen} As said in the Introduction, this part of
the proof deals with the convergence of the penalized BSDE when the jumps are driven by a general Poisson random measure. We state in Proposition \ref{prop1} that a sequence of penalized
BSDEs converges to the solution to \eqref{eq0}. To do so, we give in Section
\ref{sect:int_res} an other proof of existence of solutions to reflected BSDEs
with jumps and RCLL barriers based on the penalization method.
 We
extend the proof of \cite[Section 4]{LX07} to the case of totally inacessible
jumps. We are able to generalize their proof thanks to Mokobodzki's condition
(which in particular enables to get Lemma \ref{exist_sol}, generalizing
\cite[Lemma 4.1]{LX07}), to the comparison Theorem for BSDEs with jumps (see
Theorem \ref{QS13:th4.2} and Theorem \ref{QS14:th5.1}) and to the caracterization of the solution of the
DBBSDE as the value function of a stochastic game (proved in Proposition
\ref{stochasticgame}).

We introduce the penalization scheme, generalizing \eqref{eq_pen_scheme} to
the case of  Poisson random measure :
\begin{align}\label{eq1}
  Y^p_t=&\xi_T+\int_t^T g(s,Y^p_s,Z^p_s,U^p_s)ds + p \int_t^T (Y^p_s-\xi_s)^- ds
  -p\int_t^T (\zeta_s-Y^p_s)^- ds -\int_t^T Z^p_s dW_s \notag \\
  &-\int_t^T\int_{\R^*}U^p_s(e)\tilde{N}(ds,de)
\end{align}
with $A^p_t=p\int_0^t (Y^p_s-\xi_s)^-ds$ and $K^p_t=p\int_0^t
(\zeta_s-Y^p_s)^- ds$.

\begin{proposition}\label{prop1} Under Hypothesis \ref{hypo2}, $Y^p$ converges
  to $Y$ in $\H^2$, $Z^p$ weakly converges in $\H^2$ to $Z$, $U^p$ weakly converges in
  $\H^2_{\nu}$ to $U$, and $\alpha^p_t:=A^p_t-K^p_t$ weakly converges to
  $\alpha_t$ in $L^2(\mF_t)$. Moreover, for all $r \in [1,2[$, the following strong
  convergence holds
  \begin{align}\label{eq25}
    \lim_{p\rightarrow \infty}    \E\left[ \int_0^T |Y^p_s-Y_s|^2 ds\right]+ \E\left[ \int_0^T |Z^p_s-Z_s|^r
      ds + \int_0^T\left( \int_{\R^*} |U^p_s-U_s|^2\nu(de)
      \right)^{\frac{r}{2}}ds \right] =0.
  \end{align}
\end{proposition}
The proof of Proposition \ref{prop1} is postponed to Section \ref{sect:proof_prop1}.

%
%

\subsubsection{Intermediate result}\label{sect:int_res}  

For each $p,q$ in $\N$, since the driver
$g(s,y,z,u)+q(y-\xi_s)^--p(\zeta_s-y)^-$ is Lipschitz in $(y,z,u)$, the
following classical BSDE with jumps admits a unique solution
$(Y^{p,q},Z^{p,q},U^{p,q})$ (see \cite{TL94})

\begin{align}\label{eq3}
  Y^{p,q}_t=&\xi_T+\int_t^T g(s,Y^{p,q}_s,Z^{p,q}_s,U^{p,q}_s)ds + q \int_t^T (Y^{p,q}_s-\xi_s)^- ds
  -p\int_t^T (\zeta_s-Y^{p,q}_s)^- ds -\int_t^T Z^{p,q}_s dW_s \notag \\
  &-\int_t^T\int_{\R^*}U^{p,q}_s(e)\tilde{N}(ds,de).
\end{align}
We set $A^{p,q}_t=q\int_0^t (Y^{p,q}_s-\xi_s)^-ds$ and $K^{p,q}_t=p\int_0^t
(\zeta_s-Y^{p,q}_s)^- ds$.

\begin{theorem}\label{thm4}Let us assume that Assumption \ref{hypo2} holds. 
The quadruple $(Y^{p,q},Z^{p,q},U^{p,q},\alpha^{p,q})$, where $\alpha^{p,q}=A^{p,q}-K^{p,q}$, converges to
$(Y,Z,U,\alpha)$, the solution of \eqref{eq0}, as $p\rightarrow \infty$ then $q
\rightarrow \infty$ (or equivalently as $q\rightarrow \infty$ then $p
\rightarrow \infty$) in the following sense : $Y^{p,q}$ converges to $Y$ in
$\H^2$, $Z^{p,q}$ weakly converges to $Z$ in $\H^2$, $U^{p,q}$ weakly
converges to $U$ in $\H^2_{\nu}$, $\alpha^{p,q}_t$ weakly converges to $\alpha_t$ in
$L^2(\mF_t)$. Moreover, for each $r \in [1,2[$, the following strong
convergence holds
\begin{align}\label{eq22}
  \lim_{p\rightarrow \infty}\lim_{q\rightarrow \infty} \E \left( \int_0^T
    |Y^{p,q}_s-Y_s|^2 ds \right) + \E \left( \int_0^T
    |Z^{p,q}_s-Z_s|^r ds + \int_0^T\left( \int_{\R^*} |U^{p,q}_s-U_s|^2 \nu(de)\right)^{\frac{r}{2}}ds \right)=0.
\end{align}
\end{theorem}

The proof of Theorem \ref{thm4} is divided in several steps. We prove
\begin{enumerate}
\item the quadruple $(Y^{p,q},Z^{p,q},U^{p,q},\alpha^{p,q})$
  converges as $q\rightarrow \infty$ then $p
  \rightarrow \infty$
\item the quadruple $(Y^{p,q},Z^{p,q},U^{p,q},\alpha^{p,q})$
  converges as $p\rightarrow \infty$ then $q
  \rightarrow \infty$
\item the two limits are equal (see Lemma \ref{lem3})
\item the limit of the penalized BSDE is the solution of the reflected BSDE
  \eqref{eq0} (see Theorem \ref{thm5})
\item Equation \eqref{eq22} ensues from \eqref{eq23} and
\eqref{eq24}.\\
\end{enumerate}

{\bf Proof of point 1.}\\
Let us first state the following preliminary result.
\begin{lemma}\label{exist_sol}
Suppose that  $H, H^{\prime}\in {\cal S}^2$  are two supermartingales such that Assumption $\ref{hypo2}$ holds.
Let $ Y^*$ be the RCLL adapted process defined by  $Y^*_t:=(H_t-H'_t)\ind_{t< T}+\xi_T\ind_{t=T}$.
There exists $(Z^*, U^*,A^*, K^*) \in \mathbb{H}^2 \times \mathbb{H}^2_{\nu} \times \mathcal{A}^2 \times \mathcal{A}^2$ such that $(Y^*, Z^*, U^*,A^*, K^*)$    solves $(i), (ii), (iii)$ of  \eqref{eq0}.
\end{lemma}

\begin{proof}
By assumption, $H$ and $H'$  are square integrable supermartingales. 
The process $ Y^*$ is thus well defined. By the Doob-Meyer decomposition of supermartingales, there exist two square integrable martingales $M$ and $M'$, two square integrable nondecreasing predictable RCLL processes $V$ and $V^{'}$ with $V_0=V_0^{'}=0$ such that:
\begin{equation}\label{1}
dH_t=dM_t-dV_t \quad ; \quad dH_t^{'}=dM_t^{'}-dV_t^{'}.
\end{equation}
Define 
$$\overline{M}_t :=M_t-M_t^{'}.$$ 
 By the above relation and \eqref{1},    we derive $dY^*_t=
d\overline{M}_t-dV_t+dV_t^{'} $. Now, by the martingale representation theorem, there exist $Z^* \in \mathbb{H}^{2}, U^* \in \mathbb{H}_{\nu}^{2}$ such that:
\begin{equation}\label{3}
d\overline{M}_t=Z^*_tdW_t+\int_{\mathbb{R}^{*}}U^*_t(e)\Tilde{N}(de,dt).
\end{equation}
Consequently, \eqref{1} and \eqref{3} imply that:
\begin{align*}
  Y^*_t=&\xi_T+\int_t^Tg(s,Y_s^*,Z_s^*,U_s^*)ds-\left(\int_t^Tg(s,Y_s^*,Z_s^*,U_s^*)ds+(V_T-V_t)-(V'_T-V'_t)\right)-\int_t^TZ^*_sdW_s\\
  &-\int_t^T\int_{\mathbb{R}^{*}}U^*_s(e)\tilde{N}(ds,de).
\end{align*}
Now let $g^+$ (resp. $g^-$) denote the positive (resp. negative) part of the
function $g$. By setting $A_t^*:=V_t+\int_0^tg^+(s,Y_s^*,Z_s^*,U_s^*)ds$ and $K_t^*:=V'_t+\int_0^tg^-(s,Y_s^*,Z_s^*,U_s^*)ds$, the result follows.
\end{proof}

\begin{proposition}\label{prop4}
   Suppose Assumption \ref{hypo2} 
  holds. Then, there exists a constant $C$, independent of $p$ and $q$ such
  that we have :
  \begin{align}\label{eq12}
    \E\left[\sup_{0 \le t \le T} (Y^{p,q}_t)^2\right] + \E\left[\int_0^T
      |Z^{p,q}_t|^2 dt\right] +\E\left[\int_0^T
      \int_{\R^*}|U^{p,q}_t(e)|^2\nu(de) dt\right] + \E[(A^{p,q}_T)^2]+
    \E[(K^{p,q}_T)^2] \le C.
  \end{align}
\end{proposition}

\begin{proof}[Proof] This proof generalizes the
  proof of \cite[Proposition 4.1]{LX07} to the case of jumps. Since $p$ and $q$ play
  symmetric roles, the calculations over $p$ and $q$ are uniform throughout
  this proof.  
  From Lemma \ref{exist_sol}, we know that there exists $(Y^*,Z^*,U^*,A^*,K^*)$ in
  $\mS^2\times \H^2\times \H^2_{\nu} \times \mA^2 \times \mA^2$ such that
  \begin{align*}
    Y^*_t=\xi_T+\int_t^T g(s,\theta^*_s)ds+(A^*_T-A^*_t)-(K^*_T-K^*_t)-\int_t^T Z^*_s
    dW_s-\int_t^T \int_{\R^*} U^*_s(e)\tilde{N}(ds,de)
  \end{align*}
  and $\xi_t \le Y^*_t \le \zeta_t$ $dP \otimes dt$ a.s. ($\theta^*_s$ denotes $(Y^*_s,Z^*_s,U^*_s)$). Then, for $p,q \in
  \N$, we also have
  \begin{align*}
    Y^*_t=&\xi_T+\int_t^T g(s,\theta^*_s)ds+(A^*_T-A^*_t)-(K^*_T-K^*_t)
    +q\int_t^T (\xi_s-Y^*_s)^+ ds -p \int_t^T (Y^*_s-\zeta_s)^+ ds \\
    &-\int_t^T Z^*_s
    dW_s-\int_t^T \int_{\R^*} U^*_s(e)\tilde{N}(ds,de).
  \end{align*}

  Let $\overline{\theta}^{p,q}:=(\overline{Y}^{p,q}, \overline{Z}^{p,q},\overline{U}^{p,q})$ and $\tilde{\theta}^{p,q}=(\tilde{Y}^{p,q},
  \tilde{Z}^{p,q},\tilde{U}^{p,q})$ be the solutions of the following equations

  \begin{align}\label{eq41}
    \overline{Y}^{p,q}_t=&\xi_T+\int_t^T g(s,\overline{\theta}^{p,q}_s)ds+(A^*_T-A^*_t)
    +q\int_t^T (\xi_s-\overline{Y}^{p,q}_s)^+ ds -p \int_t^T (\overline{Y}^{p,q}_s-\zeta_s)^+ ds \\
    &-\int_t^T \overline{Z}^{p,q}_s
    dW_s-\int_t^T \int_{\R^*} \overline{U}^{p,q}_s(e)\tilde{N}(ds,de).
  \end{align}

  \begin{align}\label{eq40}
    \tilde{Y}^{p,q}_t=&\xi_T+\int_t^T g(s,\tilde{\theta}^{p,q}_s)ds-(K^*_T-K^*_t)
    +q\int_t^T (\xi_s-\tilde{Y}^{p,q}_s)^+ ds -p\int_t^T (\tilde{Y}^{p,q}_s-\zeta_s)^+ ds \\
    &-\int_t^T \tilde{Z}^{p,q}_s
    dW_s-\int_t^T \int_{\R^*} \tilde{U}^{p,q}_s(e)\tilde{N}(ds,de).
  \end{align}
By the comparison theorem for BSDEs with jumps (see Theorem \ref{QS13:th4.2}), we
get that for all $p,q$ in $\N$, $\tilde{Y}^{p,q}_t\le Y^{p,q}_t \le
\overline{Y}^{p,q}_t$, $\xi_t \le Y^*_t \le \overline{Y}^{p,q}_t$ and
$\tilde{Y}^{p,q}_t\le Y^*_t \le \zeta_t$.  Applying this result to
\eqref{eq41} gives that $(\overline{Y}^{p,q},
\overline{Z}^{p,q},\overline{U}^{p,q})$ is also solution to

\begin{align}\label{eq7}
    \overline{Y}^{p,q}_t=\xi_T+\int_t^T g(s,\overline{\theta}^{p,q}_s)ds+(A^*_T-A^*_t)
    -p \int_t^T (\overline{Y}^{p,q}_s-\zeta_s)^+ ds-\int_t^T \overline{Z}^{p,q}_s
    dW_s-\int_t^T \int_{\R^*} \overline{U}^{p,q}_s(e)\tilde{N}(ds,de).
  \end{align}
  Doing the same with \eqref{eq40} gives that $(\tilde{Y}^{p,q},
  \tilde{Z}^{p,q},\tilde{U}^{p,q})$ is also solution to
  \begin{align}\label{eq8}
    \tilde{Y}^{p,q}_t=\xi_T+\int_t^T g(s,\tilde{\theta}^{p,q}_s)ds-(K^*_T-K^*_t)
    +q\int_t^T (\xi_s-\tilde{Y}^{p,q}_s)^+ ds-\int_t^T \tilde{Z}^{p,q}_s
    dW_s-\int_t^T \int_{\R^*} \tilde{U}^{p,q}_s(e)\tilde{N}(ds,de).
  \end{align}

  Let us consider the following BSDEs
  \begin{align}
&    Y^+_t=\xi_T+\int_t^T g(s,\theta^+_s)ds +(A^*_T-A^*_t) -\int_t^T Z^+_s dW_s
    -\int_t^T \int_{\R^*} \tilde{U}^{+}_s(e)\tilde{N}(ds,de),\label{eq5}\\
 &   Y^-_t=\xi_T+\int_t^T g(s,\theta^-_s)ds -(K^*_T-K^*_t) -\int_t^T Z^-_s dW_s
    -\int_t^T \int_{\R^*} \tilde{U}^{-}_s(e)\tilde{N}(ds,de)\label{eq6},
  \end{align}
where $\theta^+_s:=(Y^+_s,Z^+_s,U^+_s)$ and $\theta^-_s:=(Y^-_s,Z^-_s,U^-_s)$.
  Since $\overline{K}^{p,q}_t:=p\int_0^t(\overline{Y}^{p,q}_s-\zeta_s)^+ ds$
  and $\tilde{A}^{p,q}_t:=q\int_0^t (\xi_s-\tilde{Y}^{p,q}_s)^+ ds$ are
  increasing processes, Theorem \ref{QS13:th4.2} applied
  to \eqref{eq7} and \eqref{eq5} (resp. to  \eqref{eq8} and \eqref{eq6}) gives
  $\overline{Y}^{p,q}_t \le Y^+_t$ (resp. $Y^-_t \le
  \tilde{Y}^{p,q}_t$). Combining theses results with the inequality  $\tilde{Y}^{p,q}_t\le Y^{p,q}_t \le
\overline{Y}^{p,q}_t$ leads to
  \begin{align}\label{eq10}
 \forall (p,q) \in \N^2,\;\forall t \in [0,T],\;\;   Y^-_t \le \tilde{Y}^{p,q}_t \le Y^{p,q}_t \le \overline{Y}^{p,q}_t \le Y^+_t.
  \end{align}
  Then we have
  \begin{align}\label{eq9}
    \E[\sup_{0 \le t \le T} (Y^{p,q}_t)^2] \le \max\{ \E[\sup_{0 \le t \le T}
    (Y^{+}_t)^2], \E[\sup_{0 \le t \le T}
    (Y^{-}_t)^2 ]\}.
  \end{align}
  Since $A^*$ and $K^*$ belong to $\mA^2$, It{\^o}'s formula, BDG inequality and
  Gronwall's Lemma
  give $\E[\sup_{0 \le t \le T}
    (Y^{+}_t)^2] \le C$ and $\E[\sup_{0 \le t \le T}
    (Y^{-}_t)^2 ] \le C$. Then we get
    \begin{align}\label{eq11}
      \E[\sup_{0 \le t \le T} (Y^{p,q}_t)^2] \le C.
    \end{align} Let us now prove that $\E[(A^{p,q}_T)^2] + \E[(K^{p,q}_T)^2]
    \le C$. Since for all $p,q$ in $\N$, $\tilde{Y}^{p,q}_t \le Y^{p,q}_t \le
    \overline{Y}^{p,q}_t$, then $\tilde{A}^{p,q}_t \ge A^{p,q}_t \ge 0$ and
    $\overline{K}^{p,q}_t \ge K^{p,q}_t \ge 0$ . It boils down to prove
    $\E[(\tilde{A}^{p,q}_T)^2] + \E[(\overline{K}^{p,q}_T)^2]\le C$. Let us
first prove that $\E[(\tilde{A}^{p,q}_T)^2] \le C$. To do so, we apply
\cite[Equation (17)]{Ess08} to \eqref{eq8} (as a sequence in $q$). In the same
way, we apply \cite[Equation (17)]{Ess08} to \eqref{eq7} (as a sequence in $p$). We get
$\E[(\overline{K}^{p,q}_T)^2]\le C$.\\

It remains to prove $ \E\left[\int_0^T
      |Z^{p,q}_t|^2 dt\right] +\E\left[\int_0^T
      \int_{\R^*}|U^{p,q}_t(e)|^2\nu(de) dt\right] \le C$. By applying It{\^o}'s formula to
    $|Y^{p,q}_t|^2$, we get
    \begin{align*}
      &\E\left[|Y^{p,q}_t|^2\right] +\E\left[\int_t^T |Z^{p,q}_s|^2 ds \right] + \E\left[\int_t^T \int_{\R^*}
    |U_s^{p,q}(e)|^2 \nu(de) ds\right] \\
  = & \E[\xi_T^2]+2\E\left[\int_t^TY^{p,q}_s
      g(s,Y^{p,q}_s,Z^{p,q}_s,U^{p,q}_s)ds\right] +2\E \left[\int_t^T
      Y^{p,q}_s q(Y^{p,q}_s-\xi_s)^- ds\right]-2\E \left[\int_t^T
      Y^{p,q}_s p(\zeta_s-Y^{p,q}_s)^- ds\right].
    \end{align*}
    The third term of the right hand side is zero if $Y^{p,q}_s \ge
    \xi_s$. Then we can bound it by $2\E \left[\sup_{0 \le t \le T} |\xi_t|
      (A^{p,q}_T-A^{p,q}_t)\right]$. The last term of the right hand side is
    bounded in the same way. We bound it by $2\E \left[\sup_{0 \le t \le T} |\zeta_t|
      (K^{p,q}_T-K^{p,q}_t)\right]$. By using that $g$ is Lipschitz, we bound the
    second term of the right hand side
    \begin{align*}
      2\E\left[\int_t^TY^{p,q}_s
      g(s,Y^{p,q}_s,Z^{p,q}_s,U^{p,q}_s)ds\right]\le 2\E\left[ \int_t^T|Y^{p,q}_s|(\|g(\cdot,0,0,0)\|_{\infty}+C_g(|Y^{p,q}_s|+|Z^{p,q}_s|+|U^{p,q}_s|))ds\right].
    \end{align*}
By applying Young's inequality, we get
\begin{align}\label{eq30}
  &\E\left[|Y^{p,q}_t|^2\right] +\E\left[\int_t^T |Z^{p,q}_s|^2 ds \right] + \E\left[\int_t^T \int_{\R^*}
    |U_s^{p,q}(e)|^2 \nu(de) ds\right] \\
  \le &\|g(\cdot,0,0,0)\|_{\infty}^2+(1+2C_g+4C_g^2)\E\left[\int_t^T |Y^{p,q}_s|^2 ds\right] +\frac{1}{2} \E\left[\int_t^T |Z^{p,q}_s|^2 ds\right] +\frac{1}{2} \E\left[\int_t^T \int_{\R^*}
    |U^{p,q}_s(e)|^2\nu(de) ds\right] \notag \\
&+  \E[\sup_{0 \le t \le T} \xi_t^2]+\E[\sup_{0
    \le t \le T} \zeta_t^2] + \E[(A^{p,q}_T)^2]+\E[(K^{p,q}_T)^2].\notag
\end{align} By combining the assumptions on $\xi$, $\zeta$, \eqref{eq11} and
the previous result bounding $\E[(A^{p,q}_T)^2]+ E[(K^{p,q}_T)^2]$, we get
$\E[\int_t^T |Z^{p,q}_s|^2 ds ] + \E[\int_t^T \int_{\R^*} |U_s^{p,q}(e)|^2
\nu(de) ds] \le C$.
\end{proof}

In \eqref{eq3}, for fixed $p$ we set
$g_p(s,y,z,u)=g(s,y,z,u)-p(\zeta_s-y)^-$. $g_p$ is Lipschitz and
\begin{align*}
  \E\left(\int_0^T (g_p(s,0,0,0))^2 ds \right) \le 2  \E\left(\int_0^T
    (g(s,0,0,0))^2 ds \right) + 2p^2 T \E (\sup_{0 \le t \le T} (\zeta_t)^2) < \infty.
\end{align*}
By Theorem \ref{QS13:th4.2}, we know that $(Y^{p,q})$ is
increasing in $q$ for all $p$. Thanks to Theorem \ref{Ess08:th4.2}, we
know that $(Y^{p,q},Z^{p,q},U^{p,q})_{q \in \N}$ has a limit
$(Y^{p,\infty},Z^{p,\infty},U^{p,\infty}):=\theta^{p,\infty}$  such that $(Y^{p,q})_q$ converges
increasingly to $Y^{p,\infty} \in \mS^2$, and thanks to Theorem \ref{Ess08:th3.1}, we know that there exists
$Z^{p,\infty} \in \H^2$, $U^{p,\infty} \in \H^2_{\nu}$ and $A^{p,\infty} \in
\mA^2$ such that
$(Y^{p,\infty},Z^{p,\infty},U^{p,\infty},A^{p,\infty})$
satisfies the following equation
\begin{align}\label{eq17}
  Y^{p,\infty}_t=&\xi_T+\int_t^T
  g(s,\theta^{p,\infty}_s)ds+(A^{p,\infty}_T-A^{p,\infty}_t)-p\int_t^T
  (\zeta_s-Y^{p,\infty}_s)^- ds -\int_t^T Z^{p,\infty}_s dW_s \notag\\
  &-\int_t^T
  \int_{\R^*} U^{p,\infty}_s(e) \tilde{N}(ds,de)
\end{align}

$Z^{p,\infty}$ is the weak limit of $(Z^{p,q})_q$ in $\H^2$, $U^{p,\infty}$ is
the weak limit of $(U^{p,q})_q$ in $\H^2_{\nu} $ and $A^{p,\infty}_t$ is the weak
limit of $(A^{p,q}_t)_q$ in $L^2(\mF_t)$. Moreover, for each $r \in [1,2[$,
the following strong convergence holds
\begin{align}\label{eq23}
  \lim_{q\rightarrow \infty} \E \left( \int_0^T
    |Y^{p,q}_s-Y^{p,\infty}_s|^2 ds \right) + \E \left( \int_0^T
    |Z^{p,q}_s-Z^{p,\infty}_s|^r ds + \int_0^T\left( \int_{\R^*} |U^{p,q}_s-U^{p,\infty}_s|^2 \nu(de)\right)^{\frac{r}{2}}ds \right)=0.
\end{align}

 From \cite[Theorem 5.1]{Ess08}, we
also get that $\forall t \in [0,T]$, $Y^{p,\infty}_t \ge \xi_t$ and $\int_0^T
(Y^{p,\infty}_{t^-}-\xi_{t^-}) dA^{p,\infty}_t=0$ a.s. Set
$K^{p,\infty}_t=p\int_0^t (\zeta_s-Y^{p,\infty}_s)^- ds$. Since $Y^{p,q}
\nearrow Y^{p,\infty}$ when $q \rightarrow \infty$, $K^{p,q}\nearrow
K^{p,\infty}$ when $q \rightarrow \infty$. By the monotone convergence theorem
and \eqref{eq12}, we get that $\E((K^{p,\infty}_T)^2) \le C$. Then we get the
following Lemma.
\begin{lemma}\label{lem1}
  There exists a constant $C$ independent of $p$ such that
  \begin{align*}
    \E\left[\sup_{0 \le t \le T} (Y^{p,\infty}_t)^2\right] + \E\left[\int_0^T
      |Z^{p,\infty}_t|^2 dt\right] +\E\left[\int_0^T
      \int_{\R^*}|U^{p,\infty}_t(e)|^2\nu(de) dt\right] + \E[(A^{p,\infty}_T)^2]+
    \E[(K^{p,\infty}_T)^2] \le C.
  \end{align*}
\end{lemma}

From Theorem \ref{QS14:th5.1}, we have $Y^{p,\infty}_t \ge Y^{p+1,\infty}_t$, then there exists a
process $Y$ such that $Y^{p,\infty} \searrow Y$. By using Fatou's lemma, we
get
\begin{align*}
  \E\left(\sup_{0 \le t \le T} (Y_t)^2 \right) \le C,
\end{align*}
and the dominated convergence theorem gives us that $\lim_{p \rightarrow
  \infty} Y^{p,\infty}=Y$ in $\H^2$. Since $(Y^{p,q})_p$ is a decreasing
sequence, $(A^{p,q})_p$ is an increasing sequence, and by passing to the limit
($(A^{p,q}_t)_q$ weakly converges to $A^{p,\infty}_t$), we get $A^{p,\infty}_t
\le A^{p+1,\infty}_t$. Then, we deduce from Lemma \ref{lem1} that there exists
a process $A$ such that $A^{p,\infty} \nearrow A$ and $\E(A^2_T) < \infty$. Since
$A^{p,q}_t-A^{p,q}_s=\int_s^t q (\xi_r-Y^{p,q}_r)^+dr \le \int_s^t q
(\xi_r-Y^{p+1,q}_r)^+dr=A^{p+1,q}_t-A^{p+1,q}_s$, we get that
\begin{align*}
  A^{p,\infty}_t-A^{p,\infty}_s \le A^{p+1,\infty}_t-A^{p+1,\infty}_s\;\; \forall\; 0\le s \le t \le T.
\end{align*}

Thanks to Lemma \ref{lem1}, we can apply the ``generalized monotonic Theorem''
\ref{thm2}: there exist $Z \in \H^2$, $U \in \H^2_{\nu}$ and $K \in \mA^2$ such that
\begin{align}\label{eq15}
  Y_t=\xi_T + \int_t^T g(s,Y_s,Z_s,U_s)ds + A_T-A_t -(K_T-K_t) -\int_t^T Z_s
  dW_s -\int_t^T\int_{\R^*} U_s(e) \tilde{N}(ds,de), 
\end{align}
$K_t$ is the weak limit of $K^{p,\infty}_t$ in $L^2(\mF_t)$, $Z$ is the weak
limit of $Z^{p,\infty}$ in $\H^2$ and $U$ is the weak
limit of $U^{p,\infty}$ in $\H^2_{\nu}$. Moreover, $A^{p,\infty}_t$ strongly
converges to $A_t$ in $L^2(\mF_t)$ and $A \in \mA^2$, and we have for each $r \in [1,2[$,
\begin{align}\label{eq24}
  \lim_{p\rightarrow \infty} \E \left( \int_0^T
    |Y^{p,\infty}_s-Y_s|^2 ds \right) + \E \left( \int_0^T
    |Z^{p,\infty}_s-Z_s|^r ds + \int_0^T\left( \int_{\R^*} |U^{p,\infty}_s-U_s|^2 \nu(de)\right)^{\frac{r}{2}}ds \right)=0.
\end{align}

{\bf Proof of point 2.}\\

Similarly, $(Y^{p,q})_p$ is decreasing for any fixed $q$. The same arguments
as before give that $(Y^{p,q},Z^{p,q},U^{p,q})_{p \in \N}$ has a limit
$(Y^{\infty,q},Z^{\infty,q},U^{\infty,q}):=\theta^{\infty,q}$ such that $(Y^{p,q})_p$ converges
decreasingly to $Y^{\infty,q} \in \mS^2$, and thanks to Theorem \ref{Ess08:th3.1}, we know that there exists
$Z^{\infty,q} \in \H^2$, $U^{\infty,q} \in \H^2_{\nu}$ and $K^{\infty,q} \in
\mA^2$ such that
$(Y^{\infty,q},Z^{\infty,q},U^{\infty,q},K^{\infty,q})$
satisfies the following equation
\begin{align}\label{eq16}
  Y^{\infty,q}_t=&\xi_T+\int_t^T
  g(s,\theta^{\infty,q}_s)ds+q\int_t^T
  (Y^{\infty,q}_s-\xi_s)^- ds-(K^{\infty,q}_T-K^{\infty,q}_t) -\int_t^T
  Z^{\infty,q}_s dW_s \notag\\
  &-\int_t^T
  \int_{\R^*} U^{\infty,q}_s(e) \tilde{N}(ds,de)
\end{align}
$Z^{\infty,q}$ is the weak limit of $(Z^{p,q})_p$ in $\H^2$, $U^{\infty,q}$ is
the weak limit of $(U^{p,q})_p$ in $\H^2_{\nu} $ and $K^{\infty,q}_t$ is the weak
limit of $(K^{p,q}_t)_p$ in $L^2(\mF_t)$. From \cite[Theorem 5.1]{Ess08}, we
also get that $\forall t \in [0,T]$, $Y^{\infty,q}_t \le \zeta_t$ and $\int_0^T
(Y^{\infty,q}_{t^-}-\zeta_{t^-}) dK^{\infty,q}_t=0$ a.s. Set
$A^{\infty,q}_t=q\int_0^t (Y^{\infty,q}_s-\xi_s)^- ds$. Since $Y^{p,q}
\searrow Y^{\infty,q}$ when $p \rightarrow \infty$, $A^{p,q}\nearrow
A^{\infty,q}$ when $p \rightarrow \infty$. By the monotone convergence theorem
and \eqref{eq12}, we get that $\E((A^{\infty,q}_T)^2) \le C$. We get the
following result, equivalent to Lemma \ref{lem1}
\begin{lemma}\label{lem2}
  There exists a constant $C$ independent of $q$ such that
  \begin{align*}
    \E\left[\sup_{0 \le t \le T} (Y^{\infty,q}_t)^2\right] + \E\left[\int_0^T
      |Z^{\infty,q}_t|^2 dt\right] +\E\left[\int_0^T
      \int_{\R^*}|U^{\infty,q}_t(e)|^2\nu(de) dt\right] + \E[(A^{\infty,q}_T)^2]+
    \E[(K^{\infty,q}_T)^2] \le C.
  \end{align*}
\end{lemma}

 From Theorem \ref{QS14:th5.1}, we have $Y^{\infty,q}_t \le Y^{\infty,q+1}_t$, then there exists a
process $Y'$ such that $Y^{\infty,q} \nearrow Y'$. By using Fatou's lemma, we
get that $Y'$ belongs to $\mS^2$, and the convergence also holds in $\H^2$. By
using the same proof as before, we can apply Theorem \ref{thm2}: there exist $Z' \in \H^2$, $U' \in \H^2_{\nu}$ and $A' \in \mA^2$ such that
\begin{align*}
  Y'_t=\xi_T + \int_t^T g(s,Y'_s,Z'_s,U'_s)ds + A'_T-A'_t -(K'_T-K'_t) -\int_t^T Z'_s
  dW_s -\int_t^T\int_{\R^*} U'_s(e) \tilde{N}(ds,de), 
\end{align*}
$A'_t$ is the weak limit of $A^{\infty,q}_t$ in $L^2(\mF_t)$, $Z'$ is the weak
limit of $Z^{\infty,q}$ in $\H^2$ and $U'$ is the weak
limit of $U^{\infty,q}$ in $\H^2_{\nu}$. Moreover, $K^{\infty,q}_t$ strongly
converges to $K'_t$ in $L^2(\mF_t)$ and $K' \in \mA^2$. We will now prove that the two limits are
equal.\\

{\bf Proof of point 3.}

\begin{lemma}\label{lem3}
  The two limits $Y$ and $Y'$ are equal. Moreover $Z=Z'$, $U=U'$ and $A-K=A'-K'$.
\end{lemma}

\begin{proof}
Since $Y^{p,q} \nearrow Y^{p,\infty}$ and  $Y^{p,q} \searrow Y^{\infty,q}$, we
get that for all $p,q \in \N$, $Y^{\infty,q} \le  Y^{p,q} \le
Y^{p,\infty}$. Then, since $Y^{p,\infty} \searrow Y$ and $Y^{\infty,q}
\nearrow Y'$, we get $Y' \le Y$. On the other hand, since $Y^{\infty,q} \le
Y^{p,q}$, we get that for all $0 \le s \le t \le T$
\begin{align*}
  A^{p,q}_t-A^{p,q}_s \le A^{\infty,q}_t - A^{\infty,q}_s.
\end{align*}
Since $(A^{p,q}_t)_q$ weakly converges to $A^{p,\infty}_t$ in $L^2(\mF_t)$,
$(A^{\infty,q}_t)_q$ weakly converges to $A'_t$ in $L^2(\mF_t)$, and $(A^{p,\infty}_t)_p$
strongly converges to $A_t$ in $L^2(\mF_t)$, taking limit in $q$ and then limit in
$p$ gives
\begin{align}\label{eq13}
  A_t-A_s \le A'_t - A'_s.
\end{align}
Since $Y^{p,q} \le Y^{p,\infty}$, we get that for all $0 \le s \le t \le T$
\begin{align*}
  K^{p,q}_t-K^{p,q}_s \le K^{p,\infty}_t - K^{p,\infty}_s.
\end{align*}
Letting $p \rightarrow \infty$ and $q\rightarrow \infty$ leads to
\begin{align}\label{eq14}
  K'_t-K'_s \le K_t - K_s.
\end{align}
Combining \eqref{eq13} and \eqref{eq14} gives that  for all $0 \le s \le t \le T$
\begin{align*}
  A_t-A_s - (K_t - K_s) \le A'_t - A'_s - (K'_t - K'_s).
\end{align*}
Thanks to Theorem \ref{QS13:th4.2}, we get that $Y' \ge
Y$. Then $Y'=Y$, and we get $Z'=Z$, $U'=U$, and $A'-K'=A-K$.

\end{proof}

{\bf Proof of point 4.}\\

It remains to prove that the limit $(Y,Z,U,A-K)$ of the penalized BSDE is the
solution of the reflected BSDE with two RCLL barriers $\xi$ and $\zeta$. To do
so, we use the links between  Dynkin games and DBBSDEs (see Proposition \ref{stochasticgame})
and Snell envelope theory (see Appendix \ref{sect:snell}).

\begin{theorem} Let $\alpha:=A-K$.
  The quartuple $(Y,Z,U,\alpha)$ solving \eqref{eq15} is the
  unique solution to \eqref{eq0}.
\end{theorem}

\begin{proof}\label{thm5} We know from Theorem \ref{thm1} that \eqref{eq0} has
  a unique solution. We already know that $(Y,Z,U,A,K)$ belongs to $\mS^2
  \times \H^2 \times \H^2_{\nu} \times \mA^2 \times \mA^2$ and satisfies
  $(ii)$. It remains to check $(iii)$ and $(iv)$. We first check $(iii)$. From
  \eqref{eq17}, we know that
  $(Y^{p,\infty},Z^{p,\infty},U^{p,\infty},A^{p,\infty})$ is the solution of
  a reflected BSDE (RBSDE in the following) with one lower barrier $\xi$. Let
  $\alpha^{p,\infty}:=A^{p,\infty}-K^{p,\infty}$. Then,
  $(Y^{p,\infty},Z^{p,\infty},U^{p,\infty},\alpha^{p,\infty})$ can be
  considered as the solution of a RBSDE with two barriers $\xi$ and
  $\zeta+(\zeta-Y^{p,\infty})^-$, since we have
\begin{align*}
  \xi \le Y^{p,\infty} \le \zeta+(\zeta-Y^{p,\infty})^-,\;\; \int_0^T (Y^{p,\infty}_t-\xi_t)dA^{p,\infty}_t=0
\end{align*}
and

\begin{align*}
  \int_0^T (Y^{p,\infty}_t-\zeta_t-(\zeta-Y^{p,\infty})_t^-)
  dK^{p,\infty}_t=-p\int_0^T (Y^{p,\infty}_t-\zeta_t)^-(\zeta_t-Y^{p,\infty}_t)^-dt=0.
\end{align*}
 From Proposition \ref{stochasticgame} we know that
\begin{align*}
  Y^{p,\infty}_t=& \essinf_{\sigma \in \mT_t} \esssup_{\tau \in
    \mT_t}\E\left(\int_t^{\sigma \wedge \tau} g(s,\theta^{p,\infty}_s)ds +
    \xi_{\tau}\ind_{\tau \le \sigma}+ \zeta_{\s}\ind_{\s<\tau}
    +(\zeta_{\s}-Y^{p,\infty}_{\s})^- \ind_{\s < \tau}\big | \mF_t \right) \\
  \ge& \essinf_{\sigma \in \mT_t} \esssup_{\tau \in
    \mT_t}\E\left(\int_t^{\sigma \wedge \tau} g(s,\theta^{p,\infty}_s)ds +
    \xi_{\tau}\ind_{\tau \le \sigma}+ \zeta_{\s}\ind_{\s<\tau}\big | \mF_t
  \right)\\
  \ge& \essinf_{\sigma \in \mT_t} \esssup_{\tau \in
    \mT_t}\E\left(\int_t^{\sigma \wedge \tau} g(s,\theta_s)ds +
    \xi_{\tau}\ind_{\tau \le \sigma}+ \zeta_{\s}\ind_{\s<\tau}\big | \mF_t
  \right)\\
  &-C_g\E \left( \int_0^T |Y^{p,\infty}_s-Y_s|+
    |Z^{p,\infty}_s-Z_s|+\|U^{p,\infty}_s-U_s\|_{\nu}ds | \mF_t
  \right).
\end{align*}
Since $Y^{p,\infty}\rightarrow Y$ in $\H^2$, $Z^{p,\infty}\rightarrow Z$ in
$\H^r$ for $r<2$, and $U^{p,\infty}\rightarrow U$ in
$\H^r_{\nu}$ for $r<2$, there exists a subsequence $p_j$ such that the last
conditional expectation converges to $0$ a.s. Taking the limit in $p$ in the
last inequality gives
\begin{align}\label{eq18}
  Y_t \ge \essinf_{\sigma \in \mT_t}  \esssup_{\tau \in
    \mT_t}\E\left(\int_t^{\sigma \wedge \tau} g(s,\theta_s)ds +
    \xi_{\tau}\ind_{\tau \le \sigma}+ \zeta_{\s}\ind_{\s<\tau}\big | \mF_t
  \right).
\end{align}

In the same way, we know that $(Y^{\infty,q},Z^{\infty,q},U^{\infty,q},K^{\infty,q})$ is the
solution of a RBSDE with one upper barrier $\zeta$. Let
$\alpha^{\infty,q}:=A^{\infty,q}-K^{\infty,q}$. Then
$(Y^{\infty,q},Z^{\infty,q},U^{\infty,q},\alpha^{\infty,q})$ is
the solution of a RBSDE with two barriers $\xi-(Y^{\infty,q}-\xi)^-$ and
$\zeta$. By Proposition \ref{stochasticgame} we know that
\begin{align*}
  Y^{\infty,q}_t\le&  \esssup_{\tau \in \mT_t} \essinf_{\sigma \in \mT_t} \E\left(\int_t^{\sigma \wedge \tau} g(s,\theta_s)ds +
    \xi_{\tau}\ind_{\tau \le \sigma}+ \zeta_{\s}\ind_{\s<\tau}\big | \mF_t
  \right)\\
  &+C_g \E \left( \int_0^T |Y^{\infty,q}_s-Y_s|+
    |Z^{\infty,q}_s-Z_s|+\|U^{\infty,q}_s-U_s\|_{\nu}ds | \mF_t
  \right).
\end{align*}
Since $Y^{\infty,q}\rightarrow Y$ in $\H^2$, $Z^{\infty,q}\rightarrow Z$ in
$\H^r$ for $r<2$, and $U^{\infty,q}\rightarrow U$ in
$\H^r_{\nu}$ for $r<2$, there exists a subsequence $q_j$ such that the last
conditional expectation converges to $0$ a.s. Taking the limit in $q$ in the
last inequality gives
\begin{align}\label{eq19}
  Y_t \le \esssup_{\tau \in \mT_t}  \essinf_{\sigma \in \mT_t} \E\left(\int_t^{\sigma \wedge \tau} g(s,\theta_s)ds +
    \xi_{\tau}\ind_{\tau \le \sigma}+ \zeta_{\s}\ind_{\s<\tau}\big | \mF_t
  \right).
\end{align}
Comparing \eqref{eq18} and \eqref{eq19} and since $ \esssup  \essinf \le
\essinf  \esssup$, we deduce
\begin{align*}
  Y_t &= \esssup_{\tau \in \mT_t}  \essinf_{\sigma \in \mT_t} \E\left(\int_t^{\sigma \wedge \tau} g(s,\theta_s)ds +
    \xi_{\tau}\ind_{\tau \le \sigma}+ \zeta_{\s}\ind_{\s<\tau}\big | \mF_t
  \right)\\
  &=  \essinf_{\sigma \in \mT_t} \esssup_{\tau \in \mT_t} \E\left(\int_t^{\sigma \wedge \tau} g(s,\theta_s)ds +
    \xi_{\tau}\ind_{\tau \le \sigma}+ \zeta_{\s}\ind_{\s<\tau}\big | \mF_t
  \right).
\end{align*}
Let $M_t:=\E(\xi_T+\int_0^T g(s,\theta_s)ds | \mF_t) - \int_0^t
g(s,\theta_s)ds$, $\tilde{\xi}_t=\xi_t - M_t$ and $\tilde{\zeta}_t=\zeta_t
-M_t$. We can rewrite $Y$ in the following form
\begin{align*}
  Y_t &= \esssup_{\tau \in \mT_t}  \essinf_{\sigma \in \mT_t}
  \E\left( \tilde{\xi}_{\tau}\ind_{\tau \le \sigma}+
    \tilde{\zeta}_{\s}\ind_{\s<\tau}\big | \mF_t
  \right)+M_t\\
  &= \essinf_{\sigma \in \mT_t} \esssup_{\tau \in \mT_t}
  \E\left(\tilde{\xi}_{\tau}\ind_{\tau \le \sigma}+
    \tilde{\zeta}_{\s}\ind_{\s<\tau}\big | \mF_t \right)+M_t
\end{align*}
Then $Y_t-M_t$ is the value of a stochastic game problem with payoff $I_t(\tau,\s)=\tilde{\xi}_{\tau}\ind_{\tau \le \sigma}+
    \tilde{\zeta}_{\s}\ind_{\s<\tau}$. Let us check that $\tilde{\xi}$ and
    $\tilde{\zeta}$ are in $\mS^2$. Since $\xi$ and $\zeta$ are in $\mS^2$, we
    only have to check that $M \in \mS^2$. Using Doob's inequality
    \begin{align*}
      \E(\sup_{0\le t \le T} (M_t)^2) &\le 2 \E \left( \sup_{0\le t \le
          T}\left(E(\xi+\int_0^T g(s,\theta_s)ds | \mF_t)\right)^2+\left(
          \int_0^T |g(s,\theta_s)|ds \right)^2\right),\\
      &\le C(1+\E \int_0^T |Y_s|^2 + |Z_s|^2 +\|U_s\|_{\nu}^2 ds ) < \infty.
    \end{align*}
    Since $\tilde{\xi}_T=\tilde{\zeta}_T=0$ and $\xi$ and $\zeta$ satisfy
    Mokobodzki'condition, we can apply \cite[Theorem 5.1]{LX07}: there
    exists a pair of non-negative RCLL supermatingales $(X^+,X^-)$ in
    $\mS^2$ such that
    \begin{align}\label{eq20}
      &X^+_t=\mR_t(X^-+\tilde{\xi}),\\
      &X^-_t=\mR_t(X^+-\tilde{\zeta})\notag
    \end{align}
    where $\mR_t(\phi)$ denotes the Snell enveloppe of $\phi$ (see Appendix \ref{sect:snell}). Thanks to
     \cite[Theorem 5.2]{LX07}, we know that $Y_t-M_t=X^+_t-X^-_t$. Moreover, by the
    Doob-Meyer decomposition theorem, we get
    \begin{align*}
      X^+_t=\E( A^{1}_T|\mF_t)-A^1_t,\; X^-_t=\E(K^1_T |\mF_t)-K^1_t 
    \end{align*} where $A^1,K^1$ are predictable increasing processes
    belonging to $\mA^2$. With the representation theorem for the martingale
    part we know that there exists $Z^1 \in \mathbb{H}^2$ and $U^1 \in
    \mathbb{H}^2_{\nu}$ such that
    \begin{align*}
      Y_t&=M_t+X^+_t-X^-_t\\
      &=\E(\xi+\int_0^T g(s,\theta_s)ds+A^1_T-K^1_T| \mF_t) -\int_0^t
      g(s,\theta_s)ds-A^1_t+K^1_t,\\
      &=Y_0+\int_0^t Z^1_s dW_s+ \int_0^t \int_{\R^*}U^1_s(e) \tilde{N}(ds,de)-\int_0^t
      g(s,\theta_s)ds-A^1_t+K^1_t. 
    \end{align*}
    Then, we compare the forward form of \eqref{eq15} and the previous
    equality, we get
    \begin{align*}
      (A_t-K_t)-(A^1_t-K^1_t)=\int_0^t (Z_s-Z^1_s) dW_s+\int_0^t \int_{\R^*}(U_s(e)-U^1_s(e)) \tilde{N}(ds,de)
    \end{align*}
and then $Z_t=Z^1_t$, $U_t=U^1_t$ and $K_t-A_t=K^1_t-A^1_t$. By
    using the properties of the Snell envelope in \eqref{eq20} (see Proposition \ref{SNELL}), we get the
    $X^+ \ge X^-+\tilde{\xi}$ and $X^- \ge X^+-\tilde{\zeta}$, which leads to
    \begin{align*}
      \xi=M+\tilde{\xi} \le Y=M+X^+-X^- \le M+\tilde{\zeta}=\zeta
    \end{align*}
    and $(iii)$ follows.\\
    It remains to check $(iv)$. By Proposition \ref{DOOB-MAYER}, we get
    that
    \begin{align*}
      0=\int_0^T(X^+_{t^-}-(\tilde{\xi}_{t^-}+X^-_{t^-}))dA^1_t=\int_0^T
      (X^+_{t^-}-X^-_{t^-}-\xi_{t^-}+M_{t^-})dA^1_t=\int_0^T (Y_{t^-}-\xi_{t^-})dA^1_t,
    \end{align*}
    and
    \begin{align*}
      0=\int_0^T(X^-_{t^-}-(X^+_{t^-}-\tilde{\zeta}_{t^-}))dK^1_t=\int_0^T
      (X^-_{t^-}-X^+_{t^-}+\zeta_{t^-}-M_{t^-})dK^1_t=\int_0^T (\zeta_{t^-}-Y_{t^-})dK^1_t,
    \end{align*}
    which ends the proof.
  \end{proof}
  
  \subsubsection{Proof of Proposition \ref{prop1}}\label{sect:proof_prop1}
In order to prove the convergence of $(Y^{p},Z^p,U^p,\alpha^p)$, we rewrite
\eqref{eq17}, the solution of the reflected BSDE with one lower obstacle $\xi$
\begin{align*}
  Y^{p,\infty}_t=&\xi+\int_t^T
  g(s,\theta^{p,\infty}_s)ds+(A^{p,\infty}_T-A^{p,\infty}_t)-p\int_t^T
  (\zeta_s-Y^{p,\infty}_s)^- ds -\int_t^T Z^{p,\infty}_s dW_s \\
  &-\int_t^T
  \int_{\R^*} U^{p,\infty}_s(e) \tilde{N}(ds,de),
\end{align*}

and \eqref{eq16}, the solution of the reflected BSDE with one upper obstacle $\zeta$

\begin{align*}
  Y^{\infty,p}_t=&\xi+\int_t^T
  g(s,\theta^{\infty,p}_s)ds+p\int_t^T
  (Y^{\infty,p}_s-\xi_s)^- ds-(K^{\infty,p}_T-K^{\infty,p}_t) -\int_t^T
  Z^{\infty,p}_s dW_s \\
  &-\int_t^T
  \int_{\R^*} U^{\infty,p}_s(e) \tilde{N}(ds,de).
\end{align*}

Since $Y^{p,\infty}_t \ge \xi_t$ and $Y^{\infty,p}\le \zeta_t$, we can
substract $p\int_t^T(Y^{p,\infty}_s-\xi_s)^-ds$ to the first BSDE and we can
add $p\int_t^T(\zeta_s-Y^{\infty,p}_s)^-ds$ to the second BSDE. By the comparison theorem we get
$Y^{\infty,p}_t\le Y^p_t \le Y^{p,\infty}_t$. Since $Y^{p,\infty} \searrow Y$
and $Y^{\infty,p} \nearrow Y$ when $p \rightarrow \infty$, we get that
$Y^{p}_t \rightarrow Y_t$ almost surely, for all $t \in [0,T]$. From
\eqref{eq24} and the corresponding result for $Y^{\infty,p}$, we get that
$\lim_{p \rightarrow \infty} \E(\int_0^T |Y^p_s - Y_s|^2 ds ) =0$.\\
Applying It{\^o}'s formula to $\E(|Y^p_t-Y_t|^2)$ between $[\sigma,\tau]$, a pair
of stopping times such that $t \le \sigma \le \tau \le T$, we get
\begin{align*}
 & \E\left(|Y^p_{\sigma}-Y_\sigma|^2 + \int_{\sigma}^{\tau} |Z^p_s-Z_s|^2 ds +
  \int_{\sigma}^{\tau} \int_{\R^*} |U^p_s(e)-U_s(e)|^2 \nu(de) ds\right) =
\E(|Y^p_{\tau}-Y_\tau|^2 )\\
&+ 2\E(\int_{\sigma}^{\tau}
  (Y^p_s-Y_s)(g(s,\theta^p_s)-g(s,\theta_s))ds) + \sum_{\sigma\le s \le \tau}
  (\Delta_s A)^2 + \sum_{\sigma\le s \le \tau}
  (\Delta_s K)^2 +2\sum_{\sigma\le s \le \tau} \Delta_s A \Delta_s K \\
  &+2\int_{\sigma}^{\tau} (Y^p_s-Y_s) d(A^p-A)_s -2\int_{\sigma}^{\tau} (Y^p_s-Y_s) d(K^p-K)_s.
\end{align*} By using the Cauchy-Schwarz inequality, the convergence of $Y^p$
to $Y$ in $\H^2$, and  the fact that $g(s,\theta^p_s)$ and $g(s,\theta_s)$ are
bounded in $L^2(\Omega \times[0,T])$, we get that the second term of the
r.h.s. tends to zero when $p$ tends to $\infty$. From the dominated convergence
theorem the last two terms of the r.h.s. also tend to zero. Since
$2\sum_{\sigma\le s \le \tau} \Delta_s A \Delta_s K \le \sum_{\sigma\le s \le
  \tau} (\Delta_s A)^2 + \sum_{\sigma\le s \le \tau} (\Delta_s K)^2$, we are
 back to Theorem \ref{Ess08:th3.1}, which ends the proof of \eqref{eq25}.\\

It remains to prove that $Z^p$ weakly converges to $Z$ in $\H^2$, $U^p$ weakly
converges to $U$ in $\H^2_{\nu}$ and $\alpha^p_t$ weakly converges to $\alpha$ in $L^2(\mF_t)$. Since
$Y^{\infty,p}_t\le Y^p_t \le Y^{p,\infty}_t$, we get $A^p_t \le
A^{\infty,p}_t$ and $K^p_t \le K^{p,\infty}_t$. Then, by using Lemmas
\ref{lem1} and \ref{lem2}, we obtain $\E((A^p_T)^2) + \E((K^p_T)^2) \le C $,
where $C$ does not depend on $p$. By applying It{\^o}'s formula to $|Y^p_t|^2$ and
by using Young's inequality as in \eqref{eq30} we get $\E(\int_0^T |Z^p_t|^2 dt + \int_0^T (\int_{\R^*}
|U^p_s(e)|^2 \nu(de) ds))\le C$, where $C$ does not depend on $p$. The
sequences $(Z^p)_{p \ge 0}$, $(U^p)_{p \ge 0}$, $(A^p_t)_{p \ge 0}$ and
$(K^p_t)_{p \ge 0}$ are bounded in the respective spaces $\H^2$, $\H^2_{\nu}$,
$L^2(\mF_t)$ and $L^2(\mF_t)$. Then, we can extract subsequences which weakly
converge in the related spaces. Let us denote $Z',U',A'$ and $K'$ the
respective limits. Since $(Z^p,U^p)$ strongly
converge to $(Z,U)$ for any $q<2$ (see \eqref{eq25}), we get that $Z=Z'$ and $U=U'$.\\

Let us prove that $A'-K'=A-K$.
We have
\begin{align*}
  A^p_t-K^p_t&=Y^p_0-Y^p_t -\int_0^t g(s,\theta^p_s)ds + \int_0^t Z^p_s dW_s
  +\int_0^t \int_{\R^*} U^p_s(e) \tilde{N}(ds,de),\\
  A_t-K_t&=Y_0-Y_t -\int_0^t g(s,\theta_s)ds + \int_0^t Z_s dW_s
  +\int_0^t \int_{\R^*} U_s(e) \tilde{N}(ds,de).\\
\end{align*}
Taking the limit in $p$ in the first equation, we get
$A'_t-K'_t=A_t-K_t$.

\section{Numerical simulations}

In this section, we illustrate the convergence of our scheme with two examples. The difficulty in the choice of examples is given by the hypothesis we assume,  in particular the Mokobodzi's condition which is difficult to check in practice.\\

\textbf{Example 1 : inaccessible jumps}\\

We consider the simulation of the solution of a DBBSDE with obstacles having only totally inaccessible jumps. More precisely, we take the barriers and driver of the following form: $\xi_t:={(W_t)}^2+\Tilde{N}_t+(T-t), \zeta_t:={(W_t)}^2+\Tilde{N}_t+3(T-t), g(t,\omega,y,z,u):=-5|y+z|+6u-1$.\\

\quad Our example satisfies the assumptions assumed in the theoretical part,
in particular Hypotheses \ref{hypo2} and
\ref{hypo6} (see Remark \ref{rem15}, point $2.$). Assumption \eqref{hypo2}, which represents the Mokobodzki's
condition, is fulfilled, since $H_t:={(W_t)}^2+\Tilde{N}_t+2(T-t)$ satisfies
$\xi_t \leq H_t \leq \zeta_t$ and $H_t=M_t+A_t$, where
$M_t:={(W_t)}^2+\Tilde{N}_t+T-t$ is a martingale and $A_t:=T-t$  is a decreasing finite variation process. \\

\quad Table \ref{tab1} gives the values of $Y_0$ with respect to parameters
$n$ and $p$ of our explicit sheme. We notice that the algorithm converges
quite fast in $p$ and $n$. However, when $n$ is too small ($n=20$ and $n=50$), the result for
$p=20000$ is quite far from the ``reference'' result ($n=600$ and
$p=20000$). Concerning the computational time, we notice that it is low, even
for big values of $p$ and $n$.

\begin{table}[htbp]
\caption{\small The solution $\overline{y}^{p,n}$ at time $t=0$}\label{tab1}
\centerline{
\begin{tabular}{|l|l|l|l|l|l|l|l|}
  \hline
$Y_0^{p,n}$& n=20 & n=50 & n=100 & n=200 & n=400& n=500& n=600\\
\hline
\hline
p=20& 1.1736 & 1.2051 & 1.2181&1.2245&1.2277&1.2283& 1.2288\\
  \hline
 p=50& 1.2077 & 1.2482 & 1.2648&1.2728&1.2767&1.2775& 1.2780\\
  \hline
p=100& 1.2214 & 1.2634 & 1.2808&1.2894&1.2936&1.2945& 1.2950\\
 \hline
p=500& 1.2350 & 1.2753 & 1.2939&1.3033&1.3079&1.3088& 1.3094\\
 \hline
p=1000& 1.2365 & 1.2767 & 1.2957&1.3051&1.3098&1.3107&1.3113\\
 \hline
p=5000&1.2376 & 1.2778 & 1.2971&1.3066&1.3113&1.3122&1.3129\\
 \hline
p=20000&1.2377 & 1.2780 & 1.2974&1.3069&1.3116&1.3125&1.3132\\
\hline
\hline
CPU time for p=20000 & 0.00071 & 0.0084  &0.0644&0.6622&6.3560&12.5970&20.0062\\
 \hline
\end{tabular}}
\end{table}

\quad  Figure \ref{img1} represents one path of
$(\overline{y}^{p,n}_t,\overline{\xi}^n_t,\zeta^n_t)_{t\ge 0}$. We notice that
for all $t$, $\overline{y}^{p,n}_t$ stays between the two obstacles.

\begin{figure}[htbp]
\centerline{
\includegraphics[width=0.75\textwidth,height=6cm]{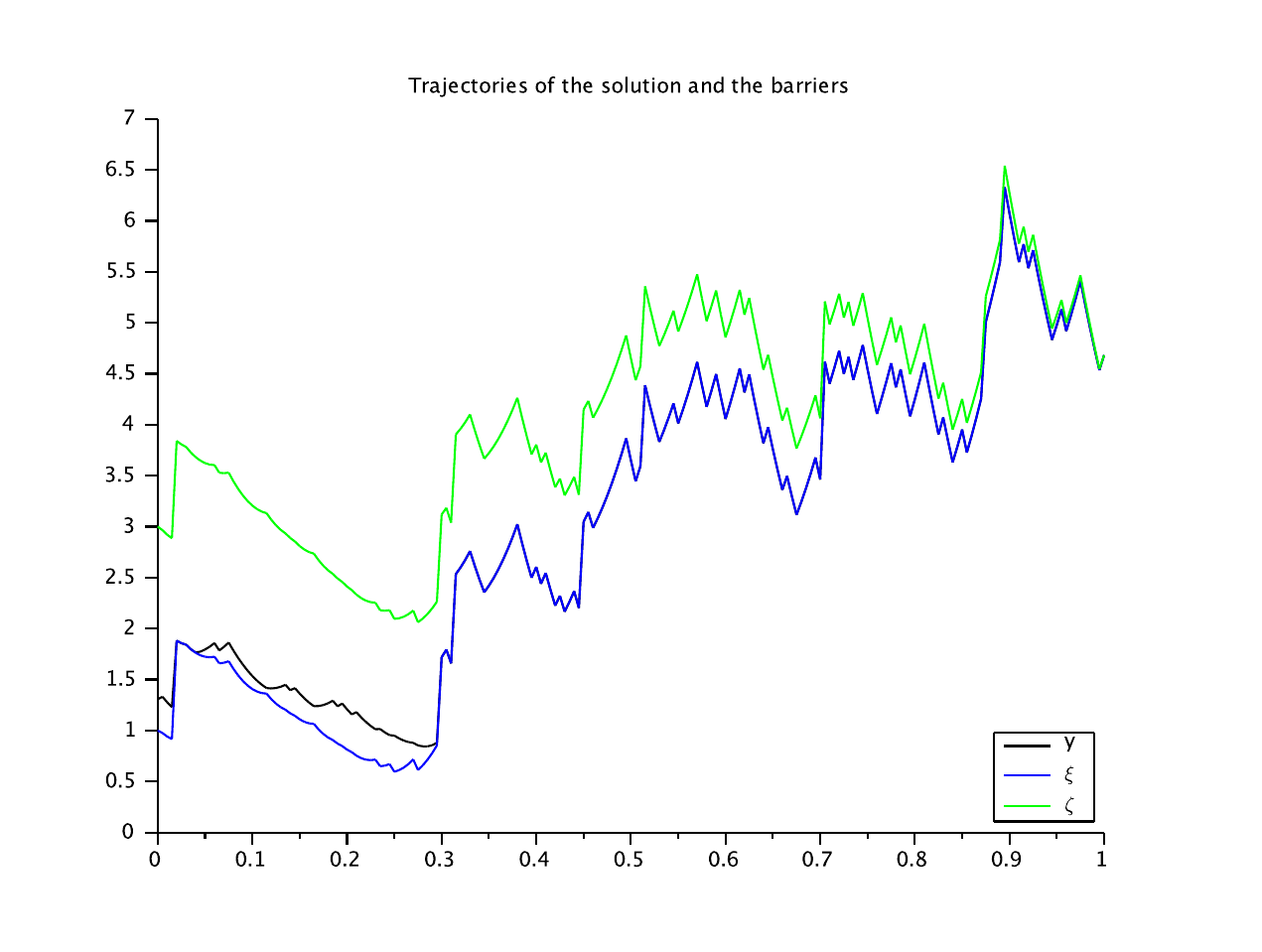}}
\caption{\small Trajectories of the solution $\overline{y}^{p,n}$ and the barriers $\overline{\xi}^n$ and $\overline{\zeta}^n$ for  $\lambda=5$, $N=200$, $p=20000$.}
\label{img1}
\end{figure}

\textbf{Example 2 : predictable and totally inaccessible jumps}\\

We consider now the simulation of the DBBSDE with obstacles having general
jumps (totally inaccessible and predictable). More precisely, we take the
barriers and driver of the following form:
$\xi_t:={(W_t)}^2+\Tilde{N}_t+(T-t)(1-\textbf{1}_{W_t \geq a}),
\zeta_t:={(W_t)}^2+\Tilde{N}_t+(T-t)(2+\textbf{1}_{W_t \geq a})$, $g(t,\omega,y,z,u):=-5|y+z|+6u-1$.  \\
We first give the numerical results for two different values of $a$, in order
to show the influence of the predictable jumps given by $\textbf{1}_{W_t \geq
  a}$ on the solution $Y$ and also the convergence in $n$ and $p$ of the
numerical explicit scheme (see Tables \ref{tab2} and \ref{tab3}).\\
\quad Then, Figures $\ref{img2}$, $\ref{img3}$ and $\ref{img4}$
allow to distinguish the predictable jumps of totally inaccesible ones and their
influence on the barriers (for e.g. the first jump of the barriers is totally
inaccessible, the second and third ones are predictable). Moreover, we remark,
as in the previous example, that the solution $Y$ stays between the two
obstacles $\xi$ and $\zeta$.

\begin{table}[htbp]
\caption{\small The solution $Y$ at time $t=0$ for a=-1}\label{tab2}
\centerline{
\begin{tabular}{|l|l|l|l|l|l|}
  \hline
 $Y_0^{p,n}$& n=100 & n=200 & n=400& n=500& n=600\\
  \hline
 p=20& 1.0745 & 1.0698& 1.0782& 1.0748& 1.0759 \\
  \hline
p=50& 1.1138 & 1.1103& 1.1191& 1.1159& 1.1170 \\
  \hline
p=100& 1.1266 & 1.1238& 1.1328& 1.1297& 1.1308 \\
  \hline
p=500& 1.1373 & 1.1353& 1.1448& 1.1419& 1.1431 \\
  \hline
p=1000& 1.1387 & 1.1369& 1.1465& 1.1437& 1.1449 \\
  \hline
p=5000& 1.1399 & 1.1382& 1.1481& 1.1453& 1.1466\\
  \hline
p=20000& 1.1401 & 1.1385& 1.1484& 1.1456& 1.1469\\
  \hline
\end{tabular}}
\end{table}

\begin{table}[htbp]
\caption{\small The solution $Y$ at time $t=0$ for a=1}\label{tab3}
\centerline{
\begin{tabular}{|l|l|l|l|l|l|}
  \hline
 $Y_0^{p,n}$& n=100 & n=200 & n=400& n=500& n=600\\
  \hline
 p=20& 1.2125 & 1.2177& 1.2203& 1.2208& 1.2212 \\
  \hline
p=50& 1.2582 & 1.2647& 1.2680& 1.2686& 1.2690 \\
  \hline
p=100& 1.2738 & 1.2808& 1.2843& 1.2850& 1.2855 \\
  \hline
p=500& 1.2866 & 1.2944& 1.2982& 1.2990& 1.2995 \\
  \hline
p=1000& 1.2884 & 1.2962& 1.3001& 1.3008& 1.3013 \\
  \hline
p=5000& 1.2898 & 1.2976 & 1.3016& 1.3023& 1.3029\\
  \hline
p=20000& 1.2900 & 1.2979& 1.3018& 1.3026& 1.3032\\
  \hline
\end{tabular}}
\end{table}

\begin{figure}[htbp]
\centerline{
\includegraphics[width=0.75\textwidth,height=5.7cm]{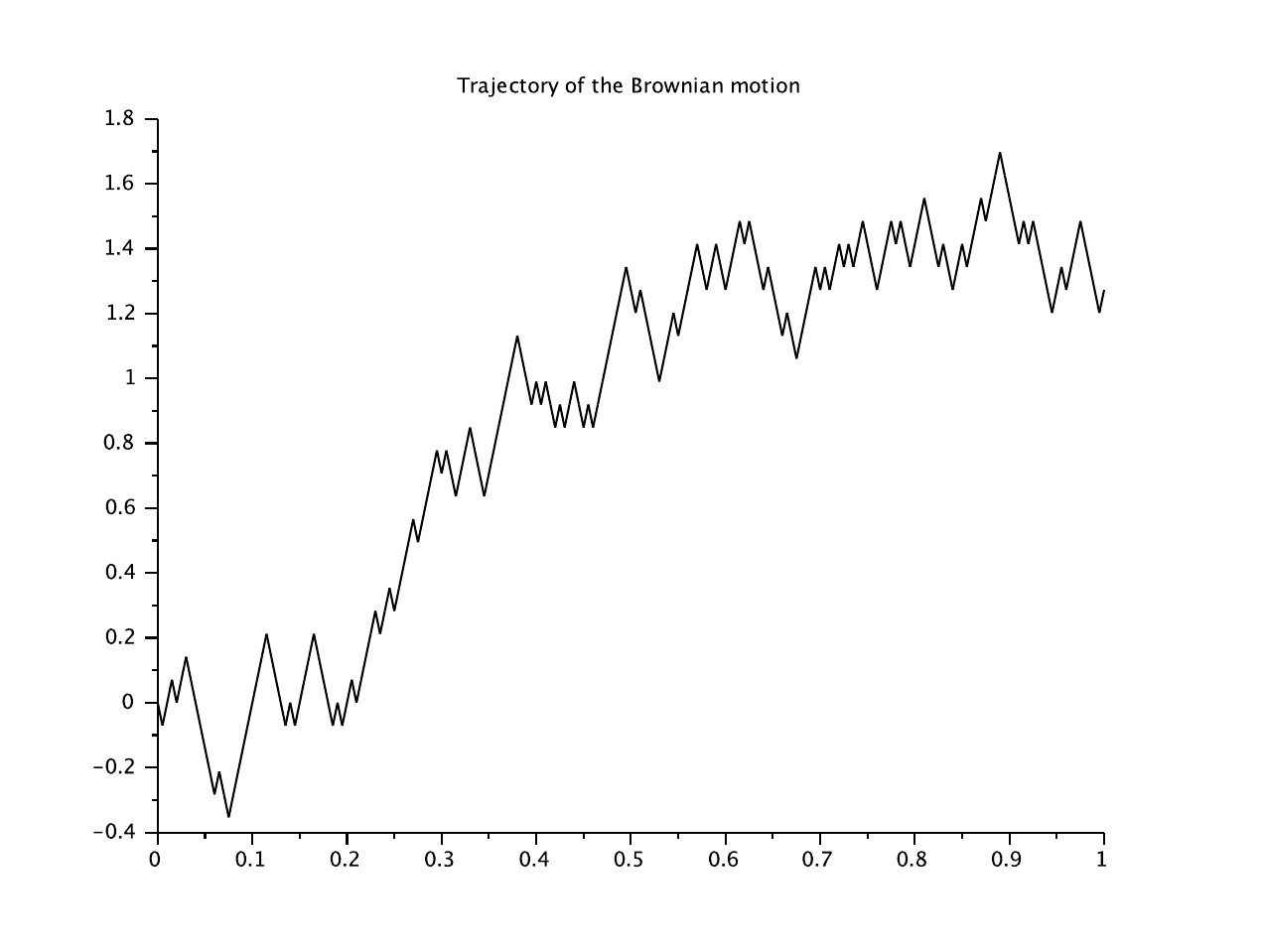}}
\caption{\small Trajectories of the Brownian motion for $a=-0.2$, $N=200$.}
\label{img2}
\end{figure}

\begin{figure}[htbp]
\centerline{
\includegraphics[width=0.75\textwidth,height=5.7cm]{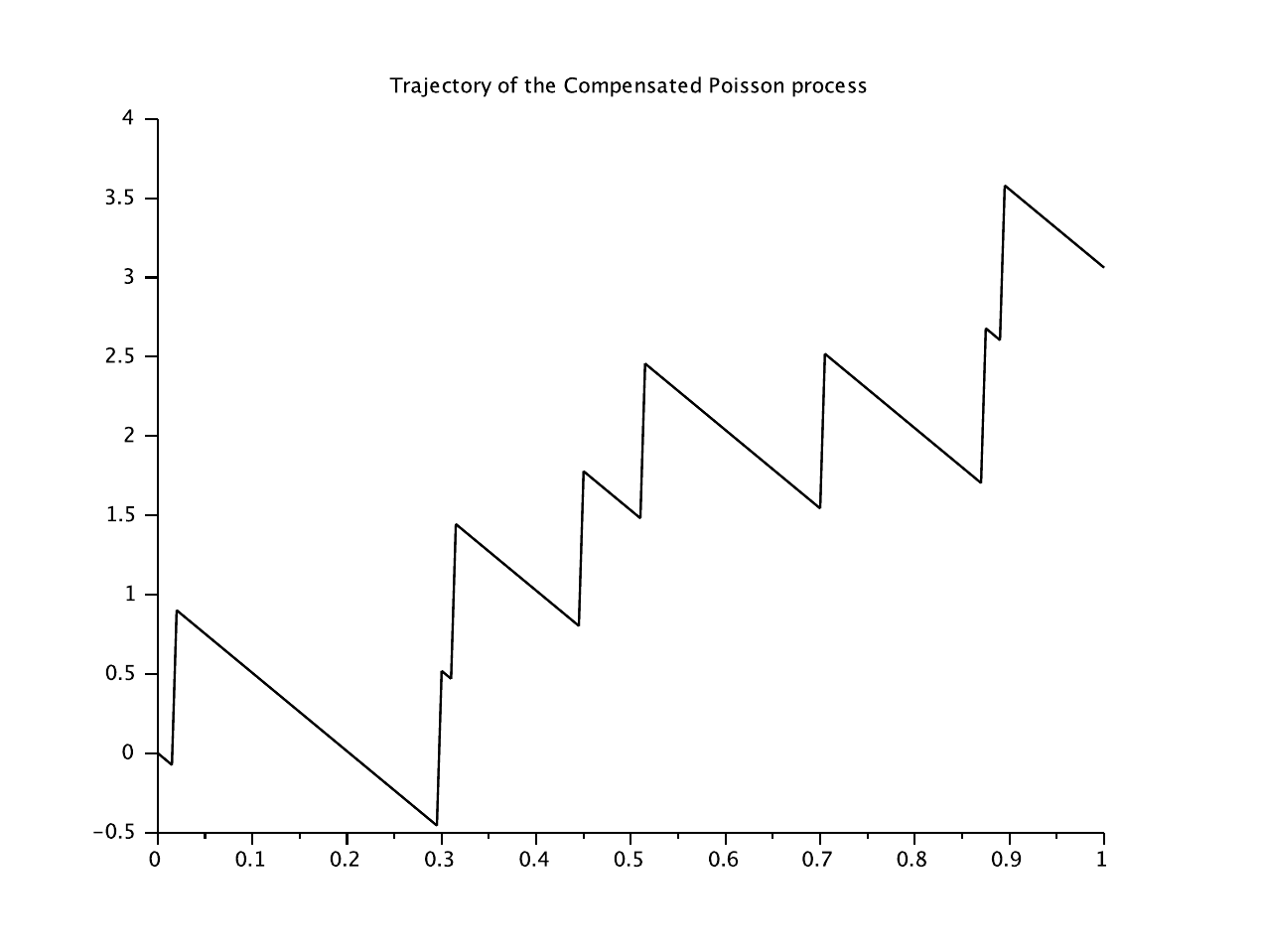}}
\caption{\small Trajectories of the  Compensated Poisson process for $\lambda=5$, $N=200$.}
\label{img3}
\end{figure}

\begin{figure}[htbp]
\centerline{
\includegraphics[width=0.75\textwidth,height=5.7cm]{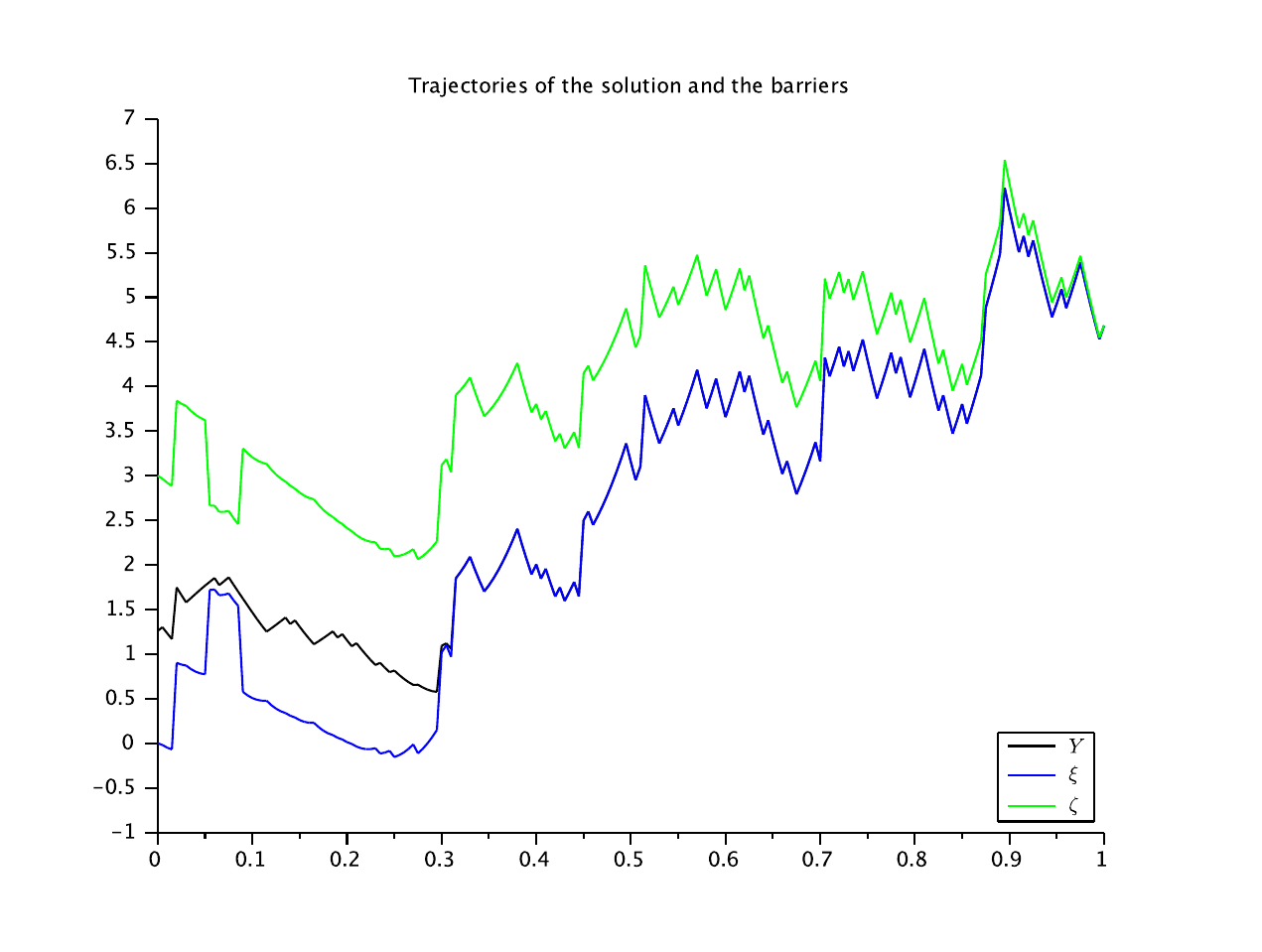}}
\caption{\small Trajectories of the solution $Y$ and the barriers $\xi$ and $\zeta$ for a=-0.2, $\lambda=5$, $N=200$.}
\label{img4}
\end{figure}

\appendix

\section{Generalized monotonic limit theorem}

The following Theorem generalizes \cite[Theorem 3.1]{PX05} and Theorem \ref{Ess08:th3.1} to the case of
doubly reflected BSDEs with jumps.

\begin{theorem}[Monotonic limit theorem]\label{thm2}
  Assume that $g$ satisfies Assumption \ref{hypo1}, and $\xi$ belongs to
  $L^2(\mF_T)$. We consider the following sequence (in $n$) of BSDEs :
  \begin{align*}
    Y^n_t=\xi+\int_t^T g(s,Y^n_s,Z^n_s,U^n_s) ds +(A^n_T-A^n_t)-(K^n_T-K^n_t)-\int_t^T Z^n_s dW_s-\int_t^T \int_{\R^*} U^n_s(e)\tilde{N}(ds,de)
  \end{align*}
  such that $Y^n \in \mS^2$, $A^n$ and $K^n$ are in $\mA^2$, and $\sup_n
  \E(\int_0^T |Z^n_s|^2 ds) +\sup_n \E(\int_0^T \int_{\R^*} |U^n_s(e)|^2
  \nu(de) ds) < \infty$. We also assume that for each $n \in \N$
  \begin{enumerate}
  \item $(A^n)_n$ is continuous and increasing and such that $A^n_0=0$ and $\sup_n \E((A^n_T)^2) < \infty$
  \item $K^j_t-K^j_s \ge K^i_t - K^i_s$, for all $0 \le s \le t \le T$ and for
    all $i \le j$
  \item for all $t \in [0,T]$, $(K^n_t)_n \nearrow K_t$ and  $E(K^2_T)<\infty$
  \item $(Y^n_t)_n$ increasingly converges to $Y_t$ with $\E(\sup_{0 \le t \le
      T} |Y_t|^2)< \infty$.
  \end{enumerate}
  Then $K \in \mA^2$ and there exist $Z \in \H^2$, $A \in \mA^2$ and $U \in \H^2_{\nu}$ such
  that 
  \begin{align*}
     Y_t=\xi+\int_t^T g(s,Y_s,Z_s,U_s) ds
    +A_T-A_t-(K_T-K_t)-\int_t^T Z_s dW_s-\int_t^T \int_{\R^*} U_s(e)\tilde{N}(ds,de).  
  \end{align*}
  $Z$ is the weak limit of $(Z^n)_n$ in $\H^2$, $K_t$ is the strong limit of
  $(K^n_t)_n$ in $L^2(\mF_t)$, $A_t$ is the weak limit of $(A^n_t)_n$
  in $L^2(\mF_t)$ and $U$ is the weak limit of $(U^n)_n$ in
  $\H^2_{\nu}$. Moreover, for all $r \in [1,2[$, the following strong
  convergence holds
  \begin{align*}
   \lim_{n \rightarrow \infty} \E\left (\int_0^T |Y^n_s-Y_s|^2 ds +\int_0^T |Z^n_s-Z_s|^r ds + \int_0^T\left(
      \int_{\R^*} |U^n_s(e)-U_s(e)|^2 \nu(de)\right)^{\frac{r}{2}}ds \right)=0.
  \end{align*}
\end{theorem}

\begin{proof}[Proof of Theorem \ref{thm2}]
  This proof follows the proofs of Theorem \ref{Ess08:th3.1} and \cite[Theorem
  3.1]{PX05}. From the hypotheses, the sequences $(Z^n)_n$, $(U^n)_n$ and
  $(g(\cdot,Y^n,Z^n,U^n))_n$ are bounded in $\H^2$, $\H^2_{\nu}$ and $L^2([0,T]
  \times \Omega)$, then we can extract subsequences which weakly converge in
  the related spaces. Let $Z$, $U$ and $g_0$ denote the respective weak limits. Thus,
  for each stopping time $\tau \le T$, the following weak convergence holds in
  $L^2(\mF_{\tau})$
  \begin{align*}
    \int_0^{\tau} g(s,Y^n_s,Z^n_s,U^n_s)ds \underset{n
      \rightarrow \infty}{\rightharpoonup} \int_0^{\tau}
    g_0(s) ds,\;\; \int_0^{\tau} Z^n_s dW_s \underset{n
      \rightarrow \infty}{\rightharpoonup} \int_0^{\tau} Z_s dW_s
  \end{align*}
  and
  \begin{align*}
    \int_0^{\tau} \int_{\R^*} U^n_s(e) \tilde{N}(ds,de) \underset{n
      \rightarrow \infty}{\rightharpoonup} \int_0^{\tau} \int_{\R^*} U_s(e)
    \tilde{N}(ds,de), \;\; K^n_{\tau}  \underset{n
      \rightarrow \infty}{\rightharpoonup} K_{\tau}
  \end{align*}
  since $(K^n_t)_n \nearrow K_t$ in $L^2(\mF_t)$.
  \begin{align*}
    A^n_{\tau}=Y^n_0-Y^n_{\tau} -\int_0^{\tau} g(s,Y^n_s,Z^n_s,U^n_s) ds +K^n_{\tau}+\int_0^{\tau} Z^n_s dW_s+\int_0^{\tau} \int_{\R^*} U^n_s(e)\tilde{N}(ds,de)
  \end{align*}
  we also have the following weak convergence in $L^2(\mF_{\tau})$
  \begin{align*}
    A^n_{\tau}\rightharpoonup A_{\tau}:=Y_0-Y_{\tau} -\int_0^{\tau} g_0(s) ds +K_{\tau}+\int_0^{\tau} Z_s dW_s+\int_0^{\tau} \int_{\R^*} U_s(e)\tilde{N}(ds,de).
  \end{align*}
  Then $\E(A_T^2)< \infty$. Since the process $(A^n_t)_t$ is increasing,
  predictable and such that $A^n_0=0$, the limit process $A$ remains an
  increasing predictable process with $A_0=0$. We deduce from \cite[Lemma
  3.2]{PX05} that $K$ is a RCLL process, and from \cite[Lemma 3.1]{PX05} that $A$ and
  $Y$ are RCLL processes. Then $Y$ has the form
  \begin{align*}
     Y_t=\xi+\int_t^T g_0(s) ds
    +A_T-A_t-(K_T-K_t)-\int_t^T Z_s dW_s-\int_t^T \int_{\R^*} U_s(e)\tilde{N}(ds,de).  
  \end{align*}
  It remains to prove that for all $r \in [1,2[$
  \begin{align*}
   \lim_{n \rightarrow \infty} \E\left (\int_0^T |Z^n_s-Z_s|^r ds + \int_0^T\left(
      \int_{\R^*} |U^n_s(e)-U_s(e)|^2 \nu(de) \right)^{\frac{r}{2}}ds \right)=0
\end{align*}
and for all $t\in[0,T]$
\begin{align*}
  \int_0^t g_0(s) ds = \int_0^t g(s,Y_s,Z_s,U_s)ds.
\end{align*}
Let $N_t=\int_0^t \int_{\R^*} U_s(e) \tilde{N}(ds,de)$ and $N^n_t=\int_0^t
\int_{\R^*} U^n_s(e) \tilde{N}(ds,de)$. We have $\Delta_s (Y^n -
Y)=\Delta_s(N^n-N+K^n - K +A)$. We appply It{\^o}'s formula to $(Y^n_t-Y_t)^2$ on
each subinterval $]\sigma,\tau]$, where $\sigma$ and $\tau$ are two
predictable stopping times such that $0 \le \sigma \le \tau \le T$. Let
$\theta^n_s$ denotes $(Y^n_s,Z^n_s,U^n_s)$
\begin{align*}
  &(Y^n_{\sigma}-Y_{\sigma})^2+ \int_{\sigma}^{\tau} |Z^n_s-Z_s|^2
  ds+\sum_{\sigma \le s \le \tau} \Delta_s  (Y^n-Y)^2\\
  &=(Y^n_{\tau}-Y_{\tau})^2+2\int_{\sigma}^{\tau}
  (Y^n_s-Y_s)(g(s,\theta^n_s)-g_0(s))ds
  +2 \int_{\sigma}^{\tau} (Y^n_s-Y_s)
  dA^n_s-2\int_{\sigma}^{\tau} (Y^n_{s^-}-Y_{s^-})
  dA_s\\
  &-2\int_{\sigma}^{\tau}(Y^n_{s^-}-Y_{s^-})d(K^n_s-K_s)
  -2\int_{\sigma}^{\tau} (Y^n_{s^-}-Y_{s^-})(Z^n_s-Z_s)dW_s-2\int_{\sigma}^{\tau}(Y^n_{s^-}-Y_{s^-})(U^n_s(e)-U_s(e))\tilde{N}(ds,de).
\end{align*}

Since $\int_{\sigma}^{\tau} (Y^n_s-Y_s) dA^n_s \le 0$,
$-2\int_{\sigma}^{\tau}(Y^n_{s^-}-Y_{s^-})d(K^n_s-K_s) \le 0$ and 
\begin{align*}
  \sum_{\sigma \le s \le \tau} \Delta_s
  (Y^n-Y)^2=\sum_{\sigma \le s \le \tau} \Delta_s (N^n-N)^2 + \sum_{\sigma \le
    s \le \tau} \Delta_s (K^n-K)^2+\sum_{\sigma \le s \le \tau} (\Delta_s
  A)^2+2 \sum_{\sigma \le s \le \tau} \Delta_s A \Delta_s (K^n-K).
\end{align*}
By taking expectation and using
$Y^n_{s^-}-Y_{s^-}=(Y^n_s-Y_s)-\Delta_s(Y^n-Y)$, we get 
\begin{align*}
  &\E(Y^n_{\sigma}-Y_{\sigma})^2+ \E\int_{\sigma}^{\tau} |Z^n_s-Z_s|^2
  ds+\E\int_{\sigma}^{\tau}\int_{\R^*} |U^n_s(e)-U_s(e)|^2\nu(de)
  ds+\E\sum_{\sigma \le s \le \tau} \Delta_s
  (K^n-K)^2\\
  &\le \E(Y^n_{\tau}-Y_{\tau})^2+2\E\int_{\sigma}^{\tau}
  (Y^n_s-Y_s)(g(s,\theta^n_s)-g_0(s))ds
  -2\E\int_{\sigma}^{\tau} (Y^n_{s}-Y_{s})dA_s+\E \sum_{\sigma \le s \le \tau}
  (\Delta_s A)^2.
\end{align*}
It comes down to \cite[Equation (10)]{Ess08}, we refer to this paper for the
end of the proof. 
\end{proof}

\section{Snell envelope theory}\label{sect:snell}
\begin{definition}
Any $\mF_t$-adapted RCLL process $\eta=(\eta_t)_{0 \le t \le T}$ is of class
$\mD[0,T]$ if the family $\{ \eta(\tau)\}_{\tau \in \mT_0}$ is uniformly integrable.
\end{definition}

\begin{definition}
  Let $\eta=(\eta_t)_{t \le T}$ be a $\mF_t$-adapted RCLL process of class
  $\mD[0,T]$. Its Snell envelope $\mR_t(\eta)$ is defined as
  \begin{align*}
    \mR_t(\eta)=\esssup_{\nu \in \mT_t} \E(\eta_{\nu} | \mF_t).
  \end{align*}
\end{definition}

\begin{proposition}\label{SNELL}
   $\mR_t(\eta)$ is the lowest RCLL $\mF_t$-supermartingale of class
   $\mD[0,T]$ which dominates $\eta$, i.e. $\P$-a.s., for all $t \in [0,T]$,
   $\mR(\eta)_t \ge \eta_t$. 
\end{proposition}

\begin{proposition}(Doob-Meyer decomposition  of Snell envelopes)\label{DOOB-MAYER}
 Let
  $\eta:=(\eta_t)_{t\le T}$ be of class $\mathcal{D}([0,T])$. There
  exists a unique decomposition of the Snell envelope
  \begin{align*}
    \mR_t(\eta)=M_t- K^c_t-K^d_t,
  \end{align*}
  where $M_t$ is a RCLL $\mF_t$-martingale, $K^c$ is a continuous integrable
  increasing process with $K^c_0=0$, and $K^d$ is a pure jump integrable
  increasing predictable RCLL process with $K^d_0=0$. Moreover, we have
  \begin{align*}
    \int_0^T (\mR_{t^-}(\eta)-\eta_{t^-}) dK_t=0,
  \end{align*}
  where $K:=K^c+K^d$.
\end{proposition}

\begin{proof}
  The first part of the proposition corresponds to the Doob-Meyer
  decomposition of supermartingales of class $\mD[0,T]$. To prove the second
  part of the proof, we write
  \begin{align*}
    \int_0^T (\mR_{t^-}(\eta)-\eta_{t^-}) dK_t=\int_0^T (\mR_{t^-}(\eta)-\eta_{t^-}) dK^d_t+\int_0^T (\mR_{t^-}(\eta)-\eta_{t^-}) dK^c_t.
  \end{align*}
  The first term of the right hand side is null, since $\{\Delta K^d >0\} \subset
  \{\mR(\eta)_-=\eta_-\}$ (see \cite[Property A.2, (ii)]{HO13}). Let us prove
  that the
  second term of the r.h.s. is also null. We know that $(\mR_t(\eta)+K^d_t)_t=(M_t- K^c_t)_t$ is a
  supermartingale satisfying $\mR_t(\eta)+K^d_t\ge \eta_t +K^d_t$, then
  $\mR_t(\eta)+K^d_t\ge \mR(\eta_t +K^d_t)$. On the other hand, for every
  supermartingale $N_t$ such that $N_t \ge \eta_t+K^d_t$, we have $N_t-K^d_t
  \ge \eta_t$, and then $N_t-K^d_t \ge \mR(\eta)_t$ (since $(N_t-K^d_t)_t$ is
  a supermartingale), then $N_t \ge \mR(\eta)_t+K^d_t$. By choosing
  $N_t:=\mR(\eta+K^d)_t$, we get $\mR_t(\eta)+K^d_t = \mR(\eta_t
  +K^d_t)$. Since $K^c$ is continuous, $(\mR_t(\eta)+K^d_t)_t$ is regular (see
  \cite[Exercise 27]{Pro05}). Then, from \cite[Property A3]{HO13}, we get that
  $\tau_t:=\inf\{s \ge t :  K^c_s-K^c_t>0\}$ is optimal after $t$. This yields
  $\int_t^{\tau_t} (\mR(\eta)_s+K^d_s-(\eta_s+K^d_s))dK^c_s=0$ for all $t\le
  T$. Then, we get $\int_0^T (\mR_{t^-}(\eta)-\eta_{t^-}) dK^c_t=0$.
\end{proof}

\section{Technical result for standard BSDEs with jumps}
\begin{lemma}\label{lem4} We assume that $\delta_n$ is small enough such that
   $(3+2p+2C_g+2C_g^2(1+\frac{1}{\lambda}e^{2\lambda T}))\delta_n<1.$ Then we have:
  \begin{align*} \sup_{j \leq n}
    \mathbb{E}[|\overline{y}_j^{p,n}|^2]+\delta_n\sum_{j=0}^{n-1}\mathbb{E}[|\overline{z}_j^{p,n}|^2]+(1-\kappa_n)\kappa_n\sum_{j=0}^{n-1}\mathbb{E}[|\overline{u}_j^{p,n}|^2]
    \leq K_{\mbox{Lem.}\ref{lem4}}.
\end{align*}
where $K_{\mbox{Lem.}\ref{lem4}}=(\|g(\cdot,0,0,0)\|^2_{\infty}+(p^2+C_gT)(\sup_n\max_j
\mathbb{E}[|\xi_j^n|^2]+\sup_n\max_j \mathbb{E}[|\zeta_j^n|^2]))e^{(3+2p+2C_g+2C_g^2(2+\frac{1}{\lambda}e^{2\lambda T}))}$.
\end{lemma}

\begin{proof}
From the explicit scheme, we derive that:
\begin{align*}
  \mathbb{E}[|\overline{y}_j^{p,n}|^2]-\mathbb{E}[|\overline{y}_{j+1}^{p,n}|^2]=&-\delta_n\mathbb{E}[|\overline{z}_j^{p,n}|^2]-(1-\kappa_n)\kappa_n\mathbb{E}[|\overline{u}_j^{p,n}|^2]-(1-\kappa_n)\kappa_n\mathbb{E}[|\overline{v}_j^{p,n}|^2]\\
  &-\delta_n^2\mathbb{E}[g_p^2(t_j, \mathbb{E}[\overline{y}_{j+1}^{p,n}|\mathcal{F}_j^n], \overline{z}_j^{p,n}, \overline{u}_j^{p,n})]+2\delta_n\mathbb{E}[\overline{y}_j^{p,n} g_p(t_j,
  \mathbb{E}[\overline{y}_{j+1}^{p,n}|\mathcal{F}_j^n], \overline{z}_j^{p,n},
  \overline{u}_j^{p,n})].
\end{align*}
Taking the sum for $j=i,...,n-1$ yields 
\begin{align*}
\mathbb{E}[|\overline{y}_i^{p,n}|^2] &\leq
\mathbb{E}[|\xi^n|^2]-\delta_n\sum_{j=i}^{n-1}\mathbb{E}[|\overline{z}_j^{p,n}|^2]-(1-\kappa_n)\kappa_n\sum_{j=i}^{n-1}\mathbb{E}[|\overline{u}_j^{p,n}|^2]+2\delta_n
\sum_{j=i}^{n-1}\mathbb{E}[\overline{y}_j^{p,n} g_p(t_j, \mathbb{E}[\overline{y}_{j+1}^{p,n}|\mathcal{F}_j^n], \overline{z}_j^{p,n}, \overline{u}_j^{p,n})]\\
&\leq
\mathbb{E}[|\xi^n|^2]-\delta_n\sum_{j=i}^{n-1}\mathbb{E}[|\overline{z}_j^{p,n}|^2]-(1-\kappa_n)\kappa_n\sum_{j=i}^{n-1}\mathbb{E}[|\overline{u}_j^{p,n}|^2]\\
&+2\delta_n\sum_{j=i}^{n-1}\mathbb{E}[|\overline{y}_j^{p,n}|(|g(t_j,0,0,0)|+C_g|\mathbb{E}[\overline{y}_{j+1}^{p,n}|\mathcal{F}_j^n]|+C_g|\overline{z}_j^{p,n}|+C_g|\overline{u}_j^{p,n}|+p(|\overline{y}_j^{p,n}|+|\xi^n_j|+|\zeta^n_j|))]
\end{align*}
Hence, we get that:
\begin{align*}
\mathbb{E}[|\overline{y}_i^{p,n}|^2]&
+\frac{\delta_n}{2}\sum_{j=i}^{n-1}\mathbb{E}[|\overline{z}_j^{p,n}|^2]+\frac{(1-\kappa_n)\kappa_n}{2}\sum_{j=i}^{n-1}\mathbb{E}[|\overline{u}_j^{p,n}|^2]
\leq\delta_n \sum_{j=i}^{n-1} \mathbb{E}[|g(t_j,0,0,0)|^2]\\
&+(p^2+C_g\delta_n)(\max_j
\mathbb{E}[|\xi_j^n|^2]+\max_j \mathbb{E}[|\zeta_j^n|^2])+\delta_n \left(3+2p+2C_g+2C_g^2+\frac{2C_g^2\delta_n}{(1-\kappa_n)\kappa_n}\right)\sum_{j=i}^{n-1}\mathbb{E}[|\overline{y}_j^{p,n}|^2].
\end{align*}
Since $\displaystyle{\frac{\delta_n}{\kappa_n(1-\kappa_n)} \le \frac{1}{\lambda} e^{2\lambda
  T}}$, the assumption on $\delta_n$ enables to apply Gronwall's Lemma, and the result follows. 
\end{proof}

\section{Some recent results on BSDEs and reflected BSDEs with jumps}
For the self-containment of the paper, we recall in this Section some recent
results used several times in the paper.

\subsection{Comparison Theorem for BSDEs  and reflected BSDEs with jumps}

\begin{theorem}[Comparison Theorem for BSDEs with jumps (\cite{QS13}, Theorem 4.2)]\label{QS13:th4.2}
  Let $\xi_1$ and $\xi_2$ be in $L^2(\mathcal{F}_T)$. Let $f_1$ be a Lipschitz
  driver and $f_2$ be a driver. For $i=1,2$ let $(X^i_t,\pi^i_t,l^i_t)$ be a
  solution in $\mathcal{S}^2\times \mathbb{H}^2 \times \mathbb{H}^2_{\nu}$ of
  the BSDE
  \begin{align}
    -dX^i_t=f_i(t,X^i_t,\pi^i_t,l^i_t) dt -\pi^i_t
    dW_t-\int_{\mathbb{R}^*}l^i_t(u)\tilde{N}(dt,du); \; X^i_T=\xi_i.
  \end{align}
  Assume that there exists a bounded predictable process $(\gamma_t)$ such
  that $dt \otimes dP\otimes \nu(du)$-a.s.
  \begin{align*}
    \gamma_t(u)\ge -1\quad \mbox{and  } |\gamma_t(u)|\le \psi(u),
  \end{align*}
  where $\psi \in L^2_{\nu}$ and such that
  \begin{align}\label{QS13:eq2}
    f_1(t,X^2_t,\pi^2_t,l^1_t)-f_1(t,X^2_t,\pi^2_t,l^2_t)\ge \langle
    \gamma_t,l^1_t-l^2_t \rangle_{\nu}, \; t\in [0,T], dt \otimes dP\; a.s.
  \end{align}
  Assume that
  \begin{align}\label{QS13:eq1}
    \xi_1 \ge \xi_2\; a.s. \mbox{  and  } f_1(t,X^2_t,\pi^2_t,l^2_t)\ge f_2(t,X^2_t,\pi^2_t,l^2_t) \; t\in [0,T], dt \otimes dP\; a.s.
  \end{align}
  Then we have
  \begin{align}\label{QS13:eq3}
    X^1_t \ge X^2_t \; a.s. \mbox{ for all } t\in [0,T].
  \end{align}
Moreover, if inequality \eqref{QS13:eq1} is satisfied for
$(X^1_t,\pi^1_t,l^1_t)$ instead of $(X^2_t,\pi^2_t,l^2_t)$ and if $f_2$
(instead of $f_1$) is Lipschitz and satisfies \eqref{QS13:eq2}, then
\eqref{QS13:eq3} still holds.
\end{theorem}

\begin{theorem}[Comparison Theorem for reflected BSDEs with jumps (\cite{QS14}, Theorem 5.1)]\label{QS14:th5.1}
Let $\xi^1, \xi^2$ be two RCLL obstacle processes in $\mathcal{S}^2$. Let $f_1$
and $f_2$ be Lipschitz drivers satisfying Assumption \ref{hypo1}. Suppose that
\begin{align*}
  &\xi^2_t\le \xi^1_t, \; 0\le t \le T \;a.s. \\
  &f_2(t,y,z,k)\le f_1(t,y,z,k),\mbox{ for all } (y,z,k)\in \mathbb{R}^2
    \times L^2_{\nu},\;dP\otimes dt\; a.s.
\end{align*}
Let $(Y^i,Z^i,k^i,A^i)$ be a solution in $\mathcal{S}^2\times \mathbb{H}^2
\times \mathbb{H}^2_{\nu} \times \mathcal{S}^2$ of
  the reflected BSDE
  \begin{align}
    &-dY^i_t=f_i(t,Y^i_t,Z^i_t,k^i_t(\cdot)) dt +dA^i_t-Z^i_t
      dW_t-\int_{\mathbb{R}^*}k^i_t(u)\tilde{N}(dt,du); \; Y^i_T=\xi^i_T,\\
    &Y^i_t \ge \xi^i_t,\;0\le t \le T\; a.s.\\    
  \end{align}
  and $A^i$ is a non decreasing RCLL predictable process with $A^i_0=0$ and
  such that
  \begin{align*}
    \int_0^T(Y^i_t-\xi^i_t)dA^{i,c}_t=0\;a.s. \mbox{ and } \Delta A^{i,d}_t=-\Delta
    Y^i_t\ind_{Y^i_{t^-}=\xi^i_{t^-}}\;a.s.
  \end{align*}
 Then $Y^2_t \le Y^1_t$ for all $t$ in
$[0,T]$ a.s. 
  \end{theorem}
  \subsection{Convergence results on reflected BSDEs with jumps}

  \begin{theorem}[Monotonic limit theorem for reflected BSDEs with jumps
    (\cite{Ess08}, Theorem 3.1)]\label{Ess08:th3.1}
    Assume that $f$ satisfies \cite[Assumption A.2]{Ess08}, $\xi \in L^2$ and
    $K^n$ is a continuous and increasing process such that $\sup_{n\in
      \mathbb{N}} \E(K^n_T)^2< \infty$ and $K^n_0=0$ for any $n\in
    \mathbb{N}$. Let $(Y^n,Z^n,V^n)$ be the solution of the following BSDE
    \begin{align*}
      Y^n_t=\xi+\int_t^Tf(s,Y^n_s,Z^n_s,V^n_s)ds+K^n_T-K^n_t-\int_t^T Z^n_s
      dW_s-\int_t^T \int_U V^n_s(u) \tilde{N}(ds,du),\;t\le T,
    \end{align*}
    where $\sup_{n\in \mathbb{N}} \E \int_0^T |Z^n_s|^2ds <\infty$ and
    $\sup_{n\in \mathbb{N}} \E \int_0^T \int_U |V^n_s(u)|^2\nu(du) ds
    <\infty$. If $Y^n$ converges increasingly to $Y$ with $\E(\sup_{0\le t \le
      T} Y_t^2)< \infty$, then there exists $Z \in \mathbb{H}^2$, $K \in
    \mathcal{A}^2$ and $V \in \mathbb{H}^2_{\nu}$ such that the triple
    $(Z,K,V)$ satisfies the following equation
    \begin{align*}
      Y_t=\xi+\int_0^T f(s,Y_s,Z_s,V_s)ds +K_T-K_t -\int_t^T Z_s dW_s
      -\int_t^T\int_U V_s(u)\tilde{N}(ds,du), \;t\le T.
    \end{align*}
    Here $Z$ is the weak limit of $(Z^n)_n$ in $\mathbb{H}^2$, $K_t$ is the weak
    limit of $(K^n_t)_n$ in $L^2(\mF_t)$ and $V$ is the weak limit of
    $(V^n)_n$ in $\mathbb{H}^2_{\nu}$. Moreover, for every $p \in [1,2[$, the
    following strong convergence holds
    \begin{align*}
      \lim_{n \rightarrow \infty} \E\left[\int_0^T |Y^n_s-Y_s|^2 ds\right] + \E\left[\int_0^T |Z^n_s-Z_s|^p
      ds + \int_0^T \left(\int_U |V^n_s(u)-V_s(u)|^2 \nu(du)\right)^{\frac{p}{2}}ds\right]=0.
    \end{align*}
  \end{theorem}

  Now we introduce the following penalized equation 
  \begin{align*}
      Y^n_t=\xi+\int_t^Tf(s,Y^n_s,Z^n_s,V^n_s)ds+K^n_T-K^n_t-\int_t^T Z^n_s
      dW_s-\int_t^T \int_U V^n_s(u) \tilde{N}(ds,du),\;t\le T,
  \end{align*}
  where $K^n_t=n\int_0^t (Y^n_s-S_s)^- ds$. We have
  \begin{theorem}[\cite{Ess08}, Theorem 4.2]\label{Ess08:th4.2}
    The sequence $(Y^n,Z^n,V^n)_n$ has a limit $(Y,Z,V)$ such that $Y^n$
    converges to $Y$ in $\mathcal{S}^2$ and $Z$ is the weak limit in
    $\mathbb{H}^2$, $K_t$ is the weak limit of $(K^n_t)_n$ in $L^2(\mF_t)$ and
    $V$ is the weak limit in $\mathbb{H}^2_{\nu}$.
  \end{theorem}
\subsection{Dynkin games and DBBSDEs}
In this section, we briefly recall the definition of a  Dynkin game,
as well as its connection with doubly reflected BSDEs, established for the
first time in \cite{CK96} in the case of a Brownian filtration and regular obstacles. This link has also been investigated  in the case of jumps and irregular obstacles (see  e.g. \cite{LX07}).

The setting of a Dynkin game is very simple. Two players observe two processes $\xi$ and $\zeta$. Player 1 chooses a stopping time $\sigma \in \mathcal{T}$, and Player 2 chooses a stopping time $\tau \in \mathcal{T}$. Player 2 pays  Player 1 the amount $I(\tau,\sigma):=\xi_{\tau \leq \sigma}+\zeta_{\sigma<\tau}$ at the stopping time $\tau \wedge \sigma$. Player 1 wishes to maximize $ \mathbb{E}[I( \tau, \sigma)]$ while Player 2 wishes to miminize it. It is then natural to define the lower and upper values of the game:
$$\overline{V}:= \inf_{\sigma \in \mathcal{T}}  \sup_{\tau \in \mathcal{T}} \mathbb{E}[I( \tau, \sigma)]; \,\, \underline{V} :=  \sup_{\tau \in \mathcal{T}} \inf_{\sigma \in \mathcal{T}}  \mathbb{E}[I( \tau, \sigma)]. $$
 The game is said to admit a value if $\overline{V}=\underline{V}$.

Let us now give the characterization of the solution of the DBBSDE as the
value function of a Dynkin game.\\
\begin{proposition}\label{stochasticgame}
Let $(Y,Z,U,\alpha) \in \mathcal{S}^2 \times \mathbb{H}^2 \times
\mathbb{H}^2_{\nu} \times \mathcal{A}^2$ be a solution of the DBBSDE \eqref{eq0}. 
For any $S \in \mathcal{T}_0$ and any stopping times $\tau, \sigma \in \mathcal{T}_S$, consider the payoff:
\begin{equation} \label{payoff}
 I_S(\tau, \sigma)= \int_S^{\tau \wedge \sigma} g(s, Y_s, Z_s, U_s(\cdot))ds+\xi_{\tau \leq \sigma}+\zeta_{\sigma<\tau}.
 \end{equation}
The upper and lower value functions at time $S$ associated to the  Dynkin game are defined respectively by 
\begin{equation} \label{overline}
\overline{V}(S) := \essinf_{\sigma \in \mathcal{T}_S }  \esssup_{\tau \in \mathcal{T}_S} \mathbb{E}[I_S( \tau, \sigma)| \cal{F}_{S}].
  \end{equation}

\begin{equation} \label{underline}
 \underline{V}(S):=\esssup_{\tau \in \mathcal{T}_S }  \essinf_{\sigma \in \mathcal{T}_S} \mathbb{E}[I_S(\tau, \sigma)| \cal{F}_{S}]
 \end{equation}
 This game has a value $V$, given by the state-process $Y$ solution of DBBSDE, i.e.
 \begin{equation}
 Y_S= \overline{V}(S)=\underline{V}(S).
 \end{equation}

\end{proposition}
Note that in the definition $\eqref{payoff}$, $(g(s, Y_s, Z_s, U_s(\cdot))_{s \leq \tau \wedge \sigma}$ represents the instantaneous reward, while $\xi_{\tau \leq \sigma}+\zeta_{\sigma<\tau}$ the terminal one.

\begin{proof}
For each $S$ $\in$ $\mathcal{T}_0$ and for each $\varepsilon >0$, let

\begin{equation}\label{epsi}
\tau^{\varepsilon}_S:= \inf \{ t \geq S,\,\,Y_t \leq \xi_t  + \varepsilon\} \quad \sigma^{\varepsilon}_S:= \inf \{ t \geq S,\,\,Y_t \geq \zeta_t  - \varepsilon\}.
\end{equation}
Remark that $\sigma^{\varepsilon}_S$  and  $\tau^{\varepsilon}_S$ $\in$ $\mathcal{T}_S$. Fix $\varepsilon >0$. We have that almost surely, if $t \in [S,
\tau^{\varepsilon}_S[$, then $ Y_t> \xi_t + \varepsilon$ and hence $Y_t> \xi_t$. It follows that the function $t \mapsto A^c_t$ is constant a.s. on $[S,
\tau^{\varepsilon}_S]$ and $t \mapsto A^d_t $ is constant a.s. on $[S, \tau^{\varepsilon}_S[$.
Also, $ Y_{ (\tau^{\varepsilon}_S)^-  } \geq \xi_{ (\tau^{\varepsilon}_S)^-  }+ \varepsilon\,$ a.s.\,
Since $\varepsilon >0$, it follows that $ Y_{ (\tau^{\varepsilon}_S)^-  } >
\xi_{ (\tau^{\varepsilon}_S)^-  }$  a.s.\,\,, which implies that $\Delta A^d _
{\tau^{\varepsilon}_S  } =0$ a.s. (see Remark \ref{rem14}). Hence, the process $A$ is constant on $[S, \tau^{\varepsilon}_S]$.
Furthermore, by the right-continuity of $(\xi_t)$ and $(Y_t)$, we clearly have
$Y_{\tau^{\varepsilon}_S} \leq \xi_{\tau^{\varepsilon}_S} + \varepsilon \quad \mbox{a.s.}$
Similarly, one can show that the process $K$ is constant on $[S, \sigma^{\varepsilon}_S]$ and that
$Y_{\sigma^{\varepsilon}_S} \geq \zeta_{\sigma^{\varepsilon}_S} - \varepsilon \quad \mbox{a.s.}$\\
Let us now consider two cases.
First, on the set $\{ \sigma^{\varepsilon}_S < \tau \}$, by using the definition of the stopping times and the fact that  $K$ is constant on $[S, \sigma^{\varepsilon}_S]$, we have:
\begin{equation}
 I_S(\tau, \sigma^{\varepsilon}_S) \leq \int_S^{\sigma^{\varepsilon}_S} g(s, Y_s, Z_s, U_s(\cdot))ds + Y_{\sigma^{\varepsilon}_S}+ \varepsilon- (K_{\sigma^{\varepsilon}_S}-K_S)+(A_{\sigma^{\varepsilon}_S}-A_S)
 \end{equation}
 \begin{equation*}
\leq Y_S+ \int_S^{\sigma^{\varepsilon}_S}Z_s dW_s + \int_S^{\sigma^{\varepsilon}_S} \int_{\mathbb{R}^*}U_s(e)\Tilde{N}(ds,de)+ \varepsilon.
\end{equation*}

On the set $\{ \tau \leq \sigma^{\varepsilon}_S  \}$, we obtain:
\begin{equation*}
 I_S(\tau, \sigma^{\varepsilon}_S) \leq \int_S^{\tau} g(s, Y_s, Z_s, U_s(\cdot))ds + Y_{\tau}- (K_{\tau}-K_S)+(A_{\tau}-A_S)
 \end{equation*}
 \begin{equation*}
\leq Y_S+ \int_S^{\tau}Z_s dW_s + \int_S^{\tau} \int_{\mathbb{R}^*}U_s(e)\Tilde{N}(ds,de).
\end{equation*}
The two above inequalities imply:
\begin{equation*}
\mathbb{E}[I_S(\tau, \sigma^{\varepsilon}_S)| \mathcal{F}_{S}] \leq Y_S+ \varepsilon.
\end{equation*}
Similarly, one can show that:
\begin{equation*}
\mathbb{E}[I_S(\tau^{\varepsilon}_S, \sigma)| \mathcal{F}_{S}] \geq Y_S- \varepsilon.
\end{equation*}
Consequently, we get that for each $\varepsilon >0$
\begin{equation*} 
 \esssup_{\tau \in \mathcal{T}_s}E[I_S(\tau,\sigma_S^{\varepsilon}) | \mathcal{F}_{S}]-\varepsilon \,\, \leq \,\, Y_S \,\, \leq \,\, \essinf_{\sigma \in \mathcal{T}_S} E[I_S(\tau_S^{\varepsilon},\sigma) | \mathcal{F}_{S}]+\varepsilon \,\, \text{ a.s.},
 \end{equation*}
that is $\overline{V}(S)-\varepsilon \,\, \leq \,\, Y_S \,\,\leq \,\, \underline{V}(S)+\varepsilon \quad \text{ a.s.}$ Since 
$\underline{V}(S) \leq \overline{V}(S)$  a.s., the result follows.

\end{proof}


\bibliographystyle{abbrv}
\bibliography{ref}

\begin{thebibliography}{10}

\bibitem{BE08}
B.~Bouchard and R.~Elie.
\newblock Discrete-time approximation of decoupled {F}orward-{B}ackward {SDE}
  with jumps.
\newblock {\em Stochastic Processes and their Applications}, (118):53--75,
  2008.

\bibitem{BDM01}
P.~Briand, B.~Delyon, and J.~M{\'e}min.
\newblock Donsker-{T}ype {T}heorem for {BSDE}s.
\newblock {\em Electron. Comm. Probab.}, (6):1--14, 2001.

\bibitem{C09}
J.-F. Chassagneux.
\newblock A discrete-time approximation for doubly reflected {BSDE}s.
\newblock {\em Adv. in Appl. Probab.}, 41(1):101--130, 2009.

\bibitem{CM08}
S.~Cr{\'e}pey and A.~Matoussi.
\newblock Reflected and doubly reflected {BSDE}s with jumps: a priori estimates
  and comparison.
\newblock {\em Ann. Appl. Probab.}, 18(5):2041--2069, 2008.

\bibitem{CK96}
J.~Cvitanic and I.~Karatzas.
\newblock Backward stochastic differential equations with reflection and dynkin
  games.
\newblock {\em The Annals of Probability}, (41):2024--2056, 1996.

\bibitem{DQS14}
R.~Dumitrescu, M.~Quenez, and A.~Sulem.
\newblock Double barrier reflected {BSDE}s with jumps and generalized dynkin
  games, \url{http://hal.upmc.fr/hal-00873688}.
\newblock 2014.

\bibitem{EKPPQ97}
N.~El~Karoui, C.~Kapoudjian, E.~Pardoux, S.~Peng, and M.~Quenez.
\newblock Reflected solutions of {B}ackward {SDE}'s and related obstacle
  problems for {PDE}'s.
\newblock {\em The Annals of Probability}, 25(2):702--737, 1997.

\bibitem{Ess08}
E.~Essaky.
\newblock Reflected backward stochastic differential equation with jumps and
  {RCLL} obstacle.
\newblock {\em Bulletin des Sciences Mathématiques}, (132):690--710, 2008.

\bibitem{EHO05}
E.~Essaky, N.~Harraj, and Y.~Ouknine.
\newblock Backward stochastic differential equation with two reflecting
  barriers and jumps.
\newblock {\em Stochastic Analysis and Applications}, 23:921--938, 2005.

\bibitem{HH06}
S.~Hamad\`ene and M.~Hassani.
\newblock {BSDE}s with two reacting barriers driven by a {B}rownian motion and
  an independent {P}oisson noise and related {D}ynkin game.
\newblock {\em Electronic Journal of Probability}, 11:121--145, 2006.

\bibitem{HO03}
S.~Hamad\`ene and Y.~Ouknine.
\newblock Reflected backward stochastic differential equation with jumps and
  random obstable.
\newblock {\em Elec. Journ. of Prob.}, 8:1--20, 2003.

\bibitem{HO13}
S.~Hamadène and Y.~Ouknine.
\newblock {Reflected Backward SDEs with general jumps,
  \url{http://arxiv.org/abs/0812.3965}}.
\newblock 2013.

\bibitem{HW09}
S.~Hamadène and H.~Wang.
\newblock {BSDE}s with two {RCLL} {R}eflecting {O}bstacles driven by a
  {B}rownian {M}otion and {P}oisson {M}easure and related {M}ixed {Z}ero-{S}um
  {G}ames.
\newblock {\em Stochastic Processes and their Applications}, 119:2881--2912,
  2009.

\bibitem{JMP89}
A.~Jakubowski, J.~M{\'e}min, and G.~Pag{\`e}s.
\newblock Convergence en loi des suites d'int{\'e}grales stochastiques sur
  l'espace $\mathbb{D}^1$ de skorokhod.
\newblock {\em Prob. Th. and Rel. Fields}, 81:111--137, 1989.

\bibitem{Kif13}
Y.~Kifer.
\newblock Dynkin games and israeli options.
\newblock {\em ISRN Probability and Statistics}, 2013.

\bibitem{LMT07}
A.~Lejay, E.~Mordecki, and S.~Torres.
\newblock {Numerical approximation of Backward Stochastic Differential
  Equations with Jumps, \url{https://hal.archives-ouvertes.fr/inria-00357992}}.
\newblock 2014.

\bibitem{LX07}
J.~Lepeltier and M.~Xu.
\newblock Reflected backward stochastic differential equations with two {RCLL}
  barriers.
\newblock {\em ESAIM: Probability and Statistics}, (11):3--22, 2007.

\bibitem{MPX02}
J.~Mémin, S.~Peng, and M.~Xu.
\newblock Convergence of solutions of discrete {R}eflected backward {SDE}'s and
  {S}imulations.
\newblock {\em Acta Mathematica Sinica}, 24(1):1--18, 2002.

\bibitem{PP90}
E.~Pardoux and S.~Peng.
\newblock Adapted solution of a backward stochastic differential equation.
\newblock {\em Systems Control Lett.}, 14(1):55--61, 1990.

\bibitem{PX05}
S.~Peng and M.~Xu.
\newblock The smallest g-supermartingale and reflected {BSDE} with single and
  double {L}$^2$ obstacles.
\newblock {\em Annales de l'Institut Henri Poincaré}, (41):605--630, 2005.

\bibitem{PX11}
S.~Peng and M.~Xu.
\newblock Numerical algorithms for bsdes with 1-d {B}rownian motion:
  convergence and simulation.
\newblock {\em ESAIM: Mathematical Modelling and Numerical Analysis},
  (45):335--360, 2011.

\bibitem{Pro05}
P.~Protter.
\newblock {\em Stochastic integration and differential equations, A new
  approach, Second Edition}, volume~21 of {\em Appl. Math.}
\newblock Springer-Verlag, Berlin Heidelberg New York, 2005.

\bibitem{QS13}
M.~Quenez and A.~Sulem.
\newblock {BSDE}s with jumps, optimization and applications to dynamic risk
  measures.
\newblock {\em Stochastic Processes and their Applications}, (123):3328--3357,
  2013.

\bibitem{QS14}
M.~Quenez and A.~Sulem.
\newblock Reflected {BSDE}s and robust optimal stopping for dynamic risk
  measures with jumps.
\newblock 2014.

\bibitem{Roy06}
M.~Royer.
\newblock Backward stochastic differential equations with jumps and related
  non-linear expectations.
\newblock {\em Stochastic Process. Appl.}, 116(10):1358--1376, 2006.

\bibitem{SLO89}
L.~S{\l}omi{\'n}ski.
\newblock Stability of strong solutions of stochastic differential equations.
\newblock {\em Stochastic Process. Appl.}, 31(2):173--202, 1989.

\bibitem{TL94}
S.~Tang and X.~Li.
\newblock Necessary conditions for optimal control of stochastic systems with
  random jumps.
\newblock {\em SIAM J. Cont. and Optim.}, 32:1447--1475, 1994.

\bibitem{Xu11}
M.~Xu.
\newblock Numerical algorithms and {S}imulations for {R}eflected {B}ackward
  {S}tochastic {D}ifferential {E}quations with {T}wo {C}ontinuous {B}arriers.
\newblock {\em Journal of Computational and Applied Mathematics},
  236:1137--1154, 2011.

\end{thebibliography}

\end{document}